\RequirePackage{fix-cm}
\documentclass{svjour3}
\usepackage{color}
\usepackage[latin1]{inputenc}
\usepackage[T1]{fontenc}
\usepackage{amsfonts}
\usepackage{amssymb}
\usepackage{amsmath}
\usepackage{stmaryrd}
\usepackage{enumerate}
\usepackage{multirow}
\usepackage{graphicx}

\usepackage{graphicx,type1cm,eso-pic,color}
\usepackage{pstricks}
\usepackage{float}
\newtheorem{assumption}[theorem]{Assumption}
\begin{document}
\setlength\arraycolsep{2pt}
\title{Nonparametric Estimation for I.I.D. Paths of a Martingale Driven Model with Application to Non-Autonomous Financial Models}
\titlerunning{Nonparametric Estimation for I.I.D. Paths of a Martingale Driven Model}
\author{Nicolas MARIE}
\institute{Laboratoire Modal'X, Universit\'e Paris Nanterre, Nanterre, France\\
\email{nmarie@parisnanterre.fr}}
\maketitle
%

% Abstract.

%
\begin{abstract}
This paper deals with a projection least squares estimator of the function $J_0$ computed from multiple independent observations on $[0,T]$ of the process $Z$ defined by $dZ_t = J_0(t)d\langle M\rangle_t + dM_t$, where $M$ is a continuous and square integrable martingale vanishing at $0$. Risk bounds are established on this estimator, on an associated adaptive estimator and on an associated discrete-time version used in practice. An appropriate transformation allows to rewrite the differential equation $dX_t = V(X_t)(b_0(t)dt +\sigma(t)dB_t)$, where $B$ is a fractional Brownian motion of Hurst parameter $H\in [1/2,1)$, as a model of the previous type. So, the second part of the paper deals with risk bounds on a nonparametric estimator of $b_0$ derived from the results on the projection least squares estimator of $J_0$. In particular, our results apply to the estimation of the drift function in a non-autonomous Black-Scholes model and to nonparametric estimation in a non-autonomous fractional stochastic volatility model.
\end{abstract}
%

% Section : Introduction.

%
\section{Introduction}\label{introduction_section}
Since the 1980's, the statistical inference for stochastic differential equations (SDE) driven by a Brownian motion has been widely investigated by many authors in the parametric and in the nonparametric frameworks. Classically, the estimators of the drift function are computed from one path of the solution to the SDE and converge when the time horizon $T > 0$ goes to infinity. The existence and the uniqueness of the stationary solution to the SDE are then required, and obtained thanks to restrictive conditions on the drift function. On such estimation methods, the reader can refer to Hoffmann \cite{HOFFMANN99} on a discrete-time wavelet based estimator, to Kutoyants \cite{KUTOYANTS04} (see Section 1.3.2 and Chapter 4) and Dalalyan \cite{DALALYAN05} on continuous-time kernel based estimators, to Comte et al. \cite{CGCR07} on a continuous-time projection least squares estimator, etc. Since few years, a new type of parametric and nonparametric estimators is investigated; those computed from multiple independent observations on $[0,T]$ of the SDE solution. On such estimation methods in the parametric framework, the reader can refer to Ditlevsen and De Gaetano \cite{DDG05} on a discrete-time maximum likelihood estimator of the drift parameter for SDE models with random effects, to Picchini, De Gaetano and Ditlevsen \cite{PDGD10} on an approximate maximum likelihood procedure for the estimation of both non-random parameters and the random effects, to Delattre and Lavielle \cite{DL13} on the SAEM algorithm combined with the extended Kalman filter to estimate the population parameters, to Delattre, Genon-Catalot and Lar\'edo \cite{DGCL13} on a discrete-time approximate maximum likelihood estimator of random effects in the drift and in the diffusion coefficients, etc. In the nonparametric framework, some copies based estimation methods of the drift function have been recently investigated. Precisely, the reader can refer to Comte and Genon-Catalot \cite{CGC20b,CGC21} on a continuous-time projection least squares estimator extended from nonparametric regression (see Cohen et al. \cite{CDL13}, Comte and Genon-Catalot \cite{CGC20a}, etc.) to diffusion processes, to Della Maestra and Hoffmann \cite{DMH21} on a continuous-time Nadaraya-Watson estimator in interacting particle systems, to Denis et al. \cite{DDM21} on a discrete-time nonparametric ridge estimator, to Marie and Rosier \cite{MR21} on both continuous-time and discrete-time versions of a Nadaraya-Watson estimator with a PCO bandwidths selection method, etc. Our paper deals with a nonparametric estimation problem of similar kind.
\\
\\
Consider the stochastic process $Z = (Z_t)_{t\in [0,T]}$, defined by
\begin{equation}\label{main_equation}
Z_t =\int_{0}^{t}J_0(s)d\langle M\rangle_s + M_t
\textrm{$;$ }
\forall t\in [0,T],
\end{equation}
where $M = (M_t)_{t\in [0,T]}\not= 0$ is a continuous and square integrable martingale vanishing at $0$, and $J_0$ is an unknown function which almost surely belongs to $\mathbb L^2([0,T],d\langle M\rangle_t)$. Under these conditions on $M$ and $J_0$, the quadratic variation $\langle M\rangle_t$ of $M$ is well-defined for any $t\in [0,T]$, and the Riemann-Stielj\`es integral of $J_0$ with respect to $s\mapsto\langle M\rangle_s$ on $[0,t]$ exists and is finite. So, the existence and the uniqueness of the process $Z$ are ensured. By assuming that $\langle M\rangle_t$ is deterministic for every $t\in [0,T]$, our paper deals with the estimator $\widehat J_{m,N}$ of $J_0$ minimizing the objective function
\begin{displaymath}
J\longmapsto
\gamma_N(J) :=
\frac{1}{N}\sum_{i = 1}^{N}\left(
\int_{0}^{T}J(s)^2d\langle M^i\rangle_s -
2\int_{0}^{T}J(s)dZ_{s}^{i}\right)
\end{displaymath}
on a $m$-dimensional function space $\mathcal S_m$, where $M^1,\dots,M^N$ (resp. $Z^1,\dots,Z^N$) are $N\in\mathbb N^*$ independent copies of $M$ (resp. $Z$) and $m\in\{1,\dots,N\}$. Precisely, risk bounds are established on $\widehat J_{m,N}$ and on the adaptive estimator $\widehat J_{\widehat m,N}$, where
\begin{displaymath}
\widehat m =\arg\min_{m\in\mathcal M_N}
\{\gamma_N(\widehat J_{m,N}) + {\rm pen}(m)\}
\end{displaymath}
with $\mathcal M_N\subset\{1,\dots,N\}$,
\begin{displaymath}
{\rm pen}(m) :=\mathfrak c_{\rm cal}\frac{m}{N}
\textrm{$;$ }
\forall m\in\mathbb N
\end{displaymath}
and $\mathfrak c_{\rm cal} > 0$ is a constant to calibrate in practice. Now, consider the differential equation
\begin{equation}\label{main_fDE}
X_t =
X_0 +\int_{0}^{t}V(X_s)(b_0(s)ds +\sigma(s)dB_s)
\textrm{$;$ }
t\in [0,T],
\end{equation}
where $X_0$ is a $\mathbb R\backslash\{0\}$-valued random variable, $B = (B_t)_{t\in [0,T]}$ is a fractional Brownian motion of Hurst parameter $H\in [1/2,1)$, the stochastic integral with respect to $B$ is taken pathwise (in Young's sense) when $H > 1/2$ and in It\^o's sense when $H = 1/2$, and $V :\mathbb R\rightarrow\mathbb R$, $\sigma : [0,T]\rightarrow\mathbb R\backslash\{0\}$ and $b_0 : [0,T]\rightarrow\mathbb R$ are at least continuous. An {\it appropriate} transformation (see Subsection \ref{subsection_auxiliary_model}) allows to rewrite Equation \eqref{main_fDE} as a model of type \eqref{main_equation} driven by the Molchan martingale which quadratic variation is $t^{2 - 2H}$ for every $t\in [0,T]$. So, our paper also deals with a risk bound on an estimator of $b_0/\sigma$ derived from $\widehat J_{m,N}$. Up to our knowledge, only Comte and Marie \cite{CM21} deals with a nonparametric estimator of the drift function computed from multiple independent observations on $[0,T]$ of the solution to a fractional SDE. Finally, applications in mathematical finance are provided. On the one hand, an estimator of the drift function in a non-autonomous Black-Scholes model is given at Subsection \ref{subsection_Black_Scholes}. On the other hand, let us consider the fractional stochastic volatility model
\begin{equation}\label{fractional_stochastic_volatility_model}
\left\{
\begin{array}{rcl}
 dS_t & = & S_t(b(t)dt +\sigma_tdW_t)\\
 d\sigma_t & = & \sigma_t(\rho_0(t)dt +\upsilon dB_t)
\end{array}\right.
\end{equation}
where $S_0$ and $\sigma_0$ are $(0,\infty)$-valued random variables, $W = (W_t)_{t\in [0,T]}$ is a Brownian motion, $\upsilon > 0$ and $b,\rho_0\in C^0(\mathbb R_+;\mathbb R)$. This is a non-autonomous extension, with fractional volatility, of the stochastic volatility model studied in Wiggins \cite{WIGGINS87}. To take $H\in [1/2,1)$ allows to take into account the persistance in volatility phenomenon (see Comte et al. \cite{CCR12}). An estimator of $\rho_0$ is given at Subsection \ref{subsection_fractional_stochastic_volatility}.
\\
\\
At Section \ref{projection_LS_estimator_section}, a detailed definition of the projection least squares estimator of $J_0$ is provided. Section \ref{risk_bounds_model_selection_section} deals with risk bounds on $\widehat J_{m,N}$, on the adaptive estimator $\widehat J_{\widehat m,N}$ and on a discrete-time version of $\widehat J_{m,N}$ used in practice. At Section \ref{section_application_fDE}, the results of Section \ref{risk_bounds_model_selection_section} on the estimator of $J_0$ are applied to the estimation of $b_0$ in Equation \eqref{main_fDE} and then to the estimator of the drift function (resp. $\rho_0$) in the non-autonomous Black-Scholes model (resp. in Equation \eqref{fractional_stochastic_volatility_model}). Finally, some numerical experiments are provided at Section \ref{section_numerical_experiments}; in Model \eqref{main_equation} when $M$ is the Molchan martingale, and in the non-autonomous Black-Scholes model.
\\
\\
{\bf Notations:}
\begin{itemize}
 \item $\langle .,.\rangle_{2,m}$ is the usual scalar product on $\mathbb R^m$, and $\|.\|_{2,m}$ is the associated norm.
 \item $\|.\|_{\rm op}$ is the spectral norm on the space $\mathcal M_m(\mathbb R)$ of $m\times m$ real matrices.
 \item For any $p\geqslant 1$, the usual norm on $\mathbb L^p([0,T],dt)$ is denoted by $\|.\|_p$.
 \item For every closed and convex subset $C$ of a Hilbert space $H$, $p_{C}^{\perp}(.)$ is the orthogonal projection from $H$ onto $C$.
 \item For every bounded function $\varphi : [0,T]\rightarrow\mathbb R$,
 \begin{displaymath}
 \|\varphi\|_{\infty,T} :=
 \sup_{t\in [0,T]}|\varphi(t)|.
 \end{displaymath}
 \item For any finite set $E$, $|E|$ is its cardinality.
\end{itemize}
%

% Section : A projection least squares estimator of the map J_0.

%
\section{A projection least squares estimator of the map $J_0$}\label{projection_LS_estimator_section}
In the sequel, the quadratic variation $\langle M\rangle = (\langle M\rangle_t)_{t\in [0,T]}$ of $M$ fulfills the following assumption.
%

% Assumption : Assumption on the quadratic variation of M.

%
\begin{assumption}\label{assumption_quadratic_variation}
The (nonnegative, increasing and continuous) process $\langle M\rangle$ is a deterministic function.
\end{assumption}
\noindent
Assumption \ref{assumption_quadratic_variation} is fulfilled by the Brownian motion and, more generally, by any martingale $(M_t)_{t\in [0,T]}$ such that
\begin{displaymath}
M_t =
\int_{0}^{t}\zeta(s)dW_s
\textrm{$;$ }
\forall t\in [0,T],
\end{displaymath}
where $W$ is a Brownian motion and $\zeta\in\mathbb L^2([0,T],dt)$. For some results, $\langle M\rangle$ fulfills the following stronger assumption.
%

% Assumption : Strong assumption on the quadratic variation of M.

%
\begin{assumption}\label{assumption_quadratic_variation_strong}
There exists $\mu\in C^0((0,T];\mathbb R_+)$ such that $\mu(.)^{-1}$ is continuous from $[0,T]$ into $\mathbb R_+$, and such that
\begin{displaymath}
\langle M\rangle_t =
\int_{0}^{t}\mu(s)ds
\textrm{$;$ }
\forall t\in [0,T].
\end{displaymath}
\end{assumption}
\noindent
Here again, Assumption \ref{assumption_quadratic_variation_strong} is fulfilled by the Brownian motion. Assumption \ref{assumption_quadratic_variation_strong} is also fulfilled by any martingale $(M_t)_{t\in [0,T]}$ such that
\begin{displaymath}
M_t =\int_{0}^{t}\zeta(s)dW_s
\textrm{$;$ }\forall t\in [0,T],
\end{displaymath}
where $W$ is a Brownian motion, $\zeta\in C^0((0,T];\mathbb R)$ and $\zeta(.)^{-1}$ is continuous from $[0,T]$ into $\mathbb R$. This last condition is satisfied, for instance, when $\zeta$ is a $(c,\infty)$-valued function with $c > 0$, or when $\zeta(t) = t^{-\kappa}$ for every $t\in (0,T]$ ($\kappa > 0$). For instance, let $M$ be the Molchan martingale defined by
\begin{displaymath}
M_t :=
\int_{0}^{t}\ell(t,s)dB_s
\textrm{$;$ }
\forall t\in [0,T],
\end{displaymath}
where $B$ is a fractional Brownian motion of Hurst parameter $H\in [1/2,1)$, and
\begin{displaymath}
\ell(t,s) :=\mathfrak c_Hs^{1/2 - H}(t - s)^{1/2 - H}
\mathbf 1_{(0,t)}(s)
\textrm{$;$ }
\forall s,t\in [0,T]
\end{displaymath}
with
\begin{displaymath}
\mathfrak c_H =
\left(\frac{\Gamma(3 - 2H)}{2H\Gamma(3/2 - H)^3\Gamma(H + 1/2)}\right)^{1/2}.
\end{displaymath}
Since
\begin{displaymath}
M_t =
(2 - 2H)^{1/2}
\int_{0}^{t}s^{1/2 - H}dW_s
\textrm{$;$ }
\forall t\in [0,T],
\end{displaymath}
where $W$ is the Brownian motion driving the Mandelbrot-Van Ness representation of the fractional Brownian motion $B$, the Molchan martingale fulfills Assumption \ref{assumption_quadratic_variation_strong} with $\mu(t) = (2 - 2H)t^{1 - 2H}$ for every $t\in (0,T]$.
%

% Subsection : The objective function.

%
\subsection{The objective function}
In order to define a least squares projection estimator of $J_0$, let us consider $N\in\mathbb N^*$ independent copies $M^1,\dots,M^N$ (resp. $Z^1,\dots,Z^N$) of $M$ (resp. $Z$), and the objective function $\gamma_N$ defined by
\begin{displaymath}
\gamma_N(J) :=
\frac{1}{N}\sum_{i = 1}^{N}\left(
\int_{0}^{T}J(s)^2d\langle M^i\rangle_s -
2\int_{0}^{T}J(s)dZ_{s}^{i}\right)
\end{displaymath}
for every $J\in\mathcal S_m$, where $m\in\{1,\dots,N\}$, $\mathcal S_m := {\rm span}\{\varphi_1,\dots,\varphi_m\}$ and $\varphi_1,\dots,\varphi_N$ are continuous functions from $[0,T]$ into $\mathbb R$ such that $(\varphi_1,\dots,\varphi_N)$ is an orthonormal family in $\mathbb L^2([0,T],dt)$.
\\
\\
{\bf Remark.} Note that since $t\in [0,T]\mapsto\langle M\rangle_t$ is nonnegative, increasing and continuous, and since the $\varphi_j$'s are continuous from $[0,T]$ into $\mathbb R$, the objective function $\gamma_N$ is well-defined.
\\
\\
For any $J\in\mathcal S_m$,
\begin{eqnarray*}
 \mathbb E[\gamma_N(J)] & = &
 \int_{0}^{T}J(s)^2d\langle M\rangle_s -
 2\int_{0}^{T}J(s)J_0(s)d\langle M\rangle_s -
 2\mathbb E\left[\int_{0}^{T}J(s)dM_s\right]\\
 & = &
 \int_{0}^{T}(J(s) - J_0(s))^2d\langle M\rangle_s -
 \int_{0}^{T}J_0(s)^2d\langle M\rangle_s.
\end{eqnarray*}
Then, the more $J$ is {\it close} to $J_0$, the more $\mathbb E[\gamma_N(J)]$ is small. For this reason, the estimator of $J_0$ minimizing $\gamma_N$ is studied in this paper.
%

% Subsection : The projection least squares estimator.

%
\subsection{The projection least squares estimator}\label{projection_LS_estimator_definition_subsection}
Consider
\begin{displaymath}
J :=
\sum_{j = 1}^{m}\theta_j\varphi_j
\quad {\rm with}\quad
\theta_1,\dots,\theta_m\in\mathbb R.
\end{displaymath}
Then,
\begin{eqnarray*}
 \nabla\gamma_N(J) & = &
 \left(\frac{1}{N}\sum_{i = 1}^{N}\left(
 2\sum_{k = 1}^{m}\theta_k\int_{0}^{T}
 \varphi_j(s)\varphi_k(s)d\langle M^i\rangle_s
 - 2\int_{0}^{T}\varphi_j(s)dZ_{s}^{i}\right)\right)_{j\in\{1,\dots,m\}}\\
 & = &
 2(\Psi_m(\theta_1,\dots,\theta_m)^* - z_{m,N})
\end{eqnarray*}
where
\begin{displaymath}
\Psi_m :=\left(\int_{0}^{T}\varphi_j(s)\varphi_k(s)d\langle M\rangle_s\right)_{j,k\in\{1,\dots,m\}}
\quad {\rm and}\quad
z_{m,N} :=\left(\frac{1}{N}\sum_{i = 1}^{N}\int_{0}^{T}\varphi_j(s)dZ_{s}^{i}\right)_{j\in\{1,\dots,m\}}.
\end{displaymath}
Moreover, the symmetric matrix $\Psi_m$ is nonnegative because under Assumption \ref{assumption_quadratic_variation},
\begin{displaymath}
u^*\Psi_mu =
\int_{0}^{T}\left(\sum_{j = 1}^{m}u_j\varphi_j(s)\right)^2d\langle M\rangle_s\geqslant 0
\end{displaymath}
for every $u\in\mathbb R^m$. In fact, since $\varphi_1,\dots,\varphi_m$ are linearly independent, $\Psi_m$ is even a positive-definite matrix, and thus $\gamma_N$ has a unique minimum in $\mathcal S_m$. This legitimates to consider the estimator
\begin{equation}\label{estimator}
\widehat J_{m,N} =\arg\min_{J\in\mathcal S_m}\gamma_N(J)
\end{equation}
of $J_0$, and since $\nabla\gamma_N(\widehat J_{m,N}) = 0$,
\begin{displaymath}
\widehat J_{m,N} =
\sum_{j = 1}^{m}\widehat\theta_j\varphi_j
\end{displaymath}
with
\begin{displaymath}
\widehat\theta_{m,N} :=
(\widehat\theta_1,\dots,\widehat\theta_m)^* =
\Psi_{m}^{-1}z_{m,N}.
\end{displaymath}
In practice, since the process $Z$ cannot be observed continuously on the time interval $[0,T]$, the vector $z_{m,N}$ has to be replaced by the approximation
\begin{displaymath}
z_{m,N,n} :=
\left(\frac{1}{N}\sum_{i = 1}^{N}\sum_{l = 0}^{n - 1}
\varphi_j(t_l)(Z_{t_{l + 1}}^{i} - Z_{t_l}^{i})\right)_{j\in\{1,\dots,m\}}
\end{displaymath}
in the definition of $\widehat J_{m,N}$, where $t_l := lT/n$ for every $l\in\{0,\dots,n\}$. This leads to the discrete-time estimator
\begin{displaymath}
\widehat J_{m,N,n} :=
\sum_{j = 1}^{m}[\widehat\theta_{m,N,n}]_j\varphi_j
\quad {\rm with}\quad
\widehat\theta_{m,N,n} :=
\Psi_{m}^{-1}z_{m,N,n}.
\end{displaymath}
%

% Section : Risk bounds and model selection.

%
\section{Risk bounds and model selection}\label{risk_bounds_model_selection_section}
In the sequel, the space $\mathbb L^2([0,T],d\langle M\rangle_t)$ is equipped with the scalar product $\langle .,.\rangle_{\langle M\rangle}$ defined by
\begin{displaymath}
\langle\varphi,\psi\rangle_{\langle M\rangle} :=
\int_{0}^{T}\varphi(s)\psi(s)d\langle M\rangle_s
\end{displaymath}
for every $\varphi,\psi\in\mathbb L^2([0,T],d\langle M\rangle_t)$. The associated norm is denoted by $\|.\|_{\langle M\rangle}$.
\\
\\
First, the following proposition provides a risk bound on $\widehat J_{m,N}$ for a fixed $m\in\{1,\dots,N\}$.
%

% Proposition : Risk bound.

%
\begin{proposition}\label{risk_bound}
Under Assumption \ref{assumption_quadratic_variation},
\begin{equation}\label{risk_bound_1}
\mathbb E[\|\widehat J_{m,N} - J_0\|_{\langle M\rangle}^{2}]
\leqslant
\min_{J\in\mathcal S_m}\|J - J_0\|_{\langle M\rangle}^{2} +\frac{2m}{N}.
\end{equation}
\end{proposition}
%

% Proof.

%
\begin{proof}
For every $J,K\in\mathcal S_m$,
\begin{eqnarray*}
 \gamma_N(J) -\gamma_N(K) & = &
 \|J\|_{\langle M\rangle}^{2} -\|K\|_{\langle M\rangle}^{2} -
 \frac{2}{N}\sum_{i = 1}^{N}\int_{0}^{T}(J(s) - K(s))dZ_{s}^{i}\\
 & = &
 \|J - J_0\|_{\langle M\rangle}^{2} -\|K - J_0\|_{\langle M\rangle}^{2} -
 \frac{2}{N}\sum_{i = 1}^{N}\int_{0}^{T}(J(s) - K(s))dM_{s}^{i}.
\end{eqnarray*}
Moreover,
\begin{displaymath}
\gamma_N(\widehat J_{m,N})\leqslant\gamma_N(J)
\textrm{$;$ }
\forall J\in\mathcal S_m.
\end{displaymath}
So,
\begin{displaymath}
\|\widehat J_{m,N} - J_0\|_{\langle M\rangle}^{2}
\leqslant
\|J - J_0\|_{\langle M\rangle}^{2} +
\frac{2}{N}\sum_{i = 1}^{N}\int_{0}^{T}(\widehat J_{m,N}(s) - J(s))dM_{s}^{i}
\end{displaymath}
for any $J\in\mathcal S_m$, and then
\begin{displaymath}
\mathbb E[\|\widehat J_{m,N} - J_0\|_{\langle M\rangle}^{2}]
\leqslant
\|J - J_0\|_{\langle M\rangle}^{2} +
2\mathbb E\left[\frac{1}{N}\sum_{i = 1}^{N}\int_{0}^{T}\widehat J_{m,N}(s)dM_{s}^{i}\right].
\end{displaymath}
Consider $j_0 = (\langle\varphi_j,J_0\rangle_{\langle M\rangle})_{j = 1,\dots,m}$, and $e = (e_1,\dots,e_m)^*$ such that
\begin{displaymath}
e_j :=
\frac{1}{N}\sum_{i = 1}^{N}\int_{0}^{T}\varphi_j(s)dM_{s}^{i}
\textrm{$;$ }
\forall j\in\{1,\dots,m\}.
\end{displaymath}
Since $e$ is a centered random vector,
\begin{eqnarray*}
 \mathbb E\left[\frac{1}{N}\sum_{i = 1}^{N}\int_{0}^{T}\widehat J_{m,N}(s)dM_{s}^{i}\right]
 & = &
 \sum_{j = 1}^{m}\mathbb E\left[
 \widehat\theta_j\cdot\frac{1}{N}\sum_{i = 1}^{N}\int_{0}^{T}\varphi_j(s)dM_{s}^{i}\right]\\
 & = &
 \mathbb E[\langle\widehat\theta,e\rangle_{2,m}] =
 \mathbb E[e^*\Psi_{m}^{-1}(j_0 + e)] =
 \mathbb E[e^*\Psi_{m}^{-1}e].
\end{eqnarray*}
Moreover, since $M_1,\dots,M_N$ are independent copies of $M$, and since $\Psi_m$ is a symmetric matrix,
\begin{eqnarray*}
 \mathbb E[e^*\Psi_{m}^{-1}e]
 & = &
 \sum_{j,k = 1}^{m}[\Psi_{m}^{-1}]_{j,k}\mathbb E[e_je_k] =
 \frac{1}{N}\sum_{j,k = 1}^{m}[\Psi_{m}^{-1}]_{j,k}
 \int_{0}^{T}\varphi_j(s)\varphi_k(s)d\langle M\rangle_s\\
 & = &
 \frac{1}{N}\sum_{k = 1}^{m}\sum_{j = 1}^{m}
 [\Psi_{m}]_{k,j}[\Psi_{m}^{-1}]_{j,k} =
 \frac{1}{N}\sum_{k = 1}^{m}[\Psi_m\Psi_{m}^{-1}]_{k,k} =
 \frac{m}{N}.
\end{eqnarray*}
Therefore,
\begin{displaymath}
\mathbb E[\|\widehat J_{m,N} - J_0\|_{\langle M\rangle}^{2}]
\leqslant
\min_{J\in\mathcal S_m}\|J - J_0\|_{\langle M\rangle}^{2} +\frac{2m}{N}.
\quad\qed
\end{displaymath}
\end{proof}
\noindent
Note that Inequality \eqref{risk_bound_1} says first that the bound on the variance of our least squares estimator of $J_0$ is of order $m/N$, as in the usual nonparametric regression framework. Under Assumption \ref{assumption_quadratic_variation_strong}, the following corollary provides a more understandable expression of the bound on the bias in Inequality \eqref{risk_bound_1}.
%

% Corollary : Weak risk bound.

%
\begin{corollary}\label{risk_bound_weak}
Under Assumption \ref{assumption_quadratic_variation_strong},
\begin{displaymath}
\mathbb E[\|\widehat J_{m,N} - J_0\|_{2}^{2}]
\leqslant
\|\mu(.)^{-1}\|_{\infty,T}
\|p_{\mathcal S_m(\mu)}^{\perp}(\mu^{1/2}J_0) -\mu^{1/2}J_0\|_{2}^{2} +
2\|\mu(.)^{-1}\|_{\infty,T}
\frac{m}{N}
\end{displaymath}
where
\begin{displaymath}
\mathcal S_m(\mu) :=
\{\iota\in\mathbb L^2([0,T],dt) :
\exists\varphi\in\mathcal S_m\textrm{, }
\forall t\in (0,T]\textrm{, }\iota(t) =\mu(t)^{1/2}\varphi(t)\}.
\end{displaymath}
\end{corollary}
%

% Proof.

%
\begin{proof}
Under Assumption \ref{assumption_quadratic_variation_strong},
\begin{displaymath}
\min_{J\in\mathcal S_m}
\|J - J_0\|_{\langle M\rangle}^{2} =
\min_{J\in\mathcal S_m}
\|\mu^{1/2}(J - J_0)\|_{2}^{2} =
\min_{\iota\in\mathcal S_m(\mu)}
\|\iota -\mu^{1/2}J_0\|_{2}^{2}
\end{displaymath}
with
\begin{displaymath}
\mathcal S_m(\mu) =
\{\iota\in\mathbb L^2([0,T],dt) :
\exists\varphi\in\mathcal S_m\textrm{, }
\forall t\in (0,T]\textrm{, }\iota(t) =\mu(t)^{1/2}\varphi(t)\}.
\end{displaymath}
Since $\mathcal S_m(\mu)$ is a closed vector subspace of $\mathbb L^2([0,T],dt)$,
\begin{equation}\label{risk_bound_weak_1}
\min_{\iota\in\mathcal S_m(\mu)}
\|\iota -\mu^{1/2}J_0\|_{2}^{2} =
\|p_{\mathcal S_m(\mu)}^{\perp}(\mu^{1/2}J_0) -\mu^{1/2}J_0\|_{2}^{2}.
\end{equation}
Moreover, since $\mu(.)^{-1}$ is continuous from $[0,T]$ into $\mathbb R_+$ under Assumption \ref{assumption_quadratic_variation_strong},
\begin{eqnarray}
 \|\widehat J_{m,N} - J_0\|_{2}^{2} & = &
 \|\mu^{-1/2}(\widehat J_{m,N} - J_0)\|_{\langle M\rangle}^{2}
 \nonumber\\
 \label{risk_bound_weak_2}
 & \leqslant &
 \|\mu(.)^{-1}\|_{\infty,T}
 \|\widehat J_{m,N} - J_0\|_{\langle M\rangle}^{2}.
\end{eqnarray}
Equality \eqref{risk_bound_weak_1} together with Inequality \eqref{risk_bound_weak_2} allow to conclude.\qed
\end{proof}
\noindent
For instance, assume that $\mathcal S_m = {\rm span}\{\overline\varphi_1,\dots,\overline\varphi_m\}$, where
\begin{displaymath}
\overline\varphi_1(t) :=
\sqrt{\frac{1}{\mu(t)T}},\quad
\overline\varphi_{2j}(t) :=
\sqrt{\frac{2}{\mu(t)T}}\cos\left(2\pi j\frac{t}{T}\right)
\quad {\rm and}\quad
\overline\varphi_{2j + 1}(t) :=
\sqrt{\frac{2}{\mu(t)T}}\sin\left(2\pi j\frac{t}{T}\right)
\end{displaymath}
for every $t\in [0,T]$ and $j\in\mathbb N^*$ satisfying $2j + 1\leqslant m$. The basis $(\varphi_1,\dots,\varphi_m)$ of $\mathcal S_m$, orthonormal in $\mathbb L^2([0,T],dt)$, is obtained from $(\overline\varphi_1,\dots,\overline\varphi_m)$ via the Gram-Schmidt process. Consider also the Sobolev space
\begin{displaymath}
\mathbb W_{2}^{\beta}([0,T]) :=
\left\{\iota\in C^{\beta - 1}([0,T];\mathbb R) :
\int_{0}^{T}\iota^{(\beta)}(t)^2dt <\infty\right\}
\textrm{$;$ }\beta\in\mathbb N^*,
\end{displaymath}
and assume that there exists $\iota_0\in\mathbb W_{2}^{\beta}([0,T])$ such that $\iota_0(t) =\mu(t)^{1/2}J_0(t)$ for every $t\in (0,T]$. Then, by DeVore and Lorentz \cite{DL93}, Chapter 7, Corollary 2.4, there exists a constant $\mathfrak c_{\beta,T} > 0$, not depending on $m$, such that
\begin{displaymath}
\|p_{\mathcal S_m(\mu)}^{\perp}(\mu^{1/2}J_0) -\mu^{1/2}J_0\|_{2}^{2} =
\|p_{\mathcal S_m(\mu)}^{\perp}(\iota_0) -\iota_0\|_{2}^{2}\leqslant
\mathfrak c_{\beta,T}m^{-2\beta}.
\end{displaymath}
Therefore, by Corollary \ref{risk_bound_weak},
\begin{displaymath}
\mathbb E[\|\widehat J_{m,N} - J_0\|_{2}^{2}]
\leqslant
\|\mu(.)^{-1}\|_{\infty,T}\left(\mathfrak c_{\beta,T}m^{-2\beta}
+\frac{2m}{N}\right).
\end{displaymath}
Now, consider $m_N\in\{1,\dots,N\}$, $\mathcal M_N :=\{1,\dots,m_N\}$ and
\begin{equation}\label{model_selection_method}
\widehat m =\arg\min_{m\in\mathcal M_N}
\{\gamma_N(\widehat J_{m,N}) + {\rm pen}(m)\}
\quad {\rm with}\quad
{\rm pen}(.) :=\mathfrak c_{\rm cal}\frac{.}{N},
\end{equation}
where $\mathfrak c_{\rm cal} > 0$ is a constant to calibrate in practice via, for instance, the {\it slope heuristic}. In the sequel, the $\varphi_j$'s fulfill the following assumption.
%

% Assumption : Nested spaces.

%
\begin{assumption}\label{assumption_nested}
For every $m,m'\in\{1,\dots,N\}$, if $m > m'$, then $\mathcal S_{m'}\subset\mathcal S_m$.
\end{assumption}
\noindent
The following theorem provides a risk bound on the adaptive estimator $\widehat J_{\widehat m,N}$.
%

% Theorem : Risk bound adaptive estimator.

%
\begin{theorem}\label{risk_bound_adaptive_estimator}
Under Assumptions \ref{assumption_quadratic_variation} and \ref{assumption_nested}, there exists a deterministic constant $\mathfrak c_{\ref{risk_bound_adaptive_estimator}} > 0$, not depending on $N$, such that
\begin{displaymath}
\mathbb E[\|\widehat J_{\widehat m,N} - J_0\|_{\langle M\rangle}^{2}]
\leqslant
\mathfrak c_{\ref{risk_bound_adaptive_estimator}}\left(
\min_{m\in\mathcal M_N}\{
\mathbb E[\|\widehat J_{m,N} - J_0\|_{\langle M\rangle}^{2}] + {\rm pen}(m)\} +
\frac{1}{N}\right).
\end{displaymath}
Moreover, under Assumption \ref{assumption_quadratic_variation_strong},
\begin{displaymath}
\mathbb E[\|\widehat J_{\widehat m,N} - J_0\|_{2}^{2}]
\leqslant
\mathfrak c_{\ref{risk_bound_adaptive_estimator}}
\|\mu(.)^{-1}\|_{\infty,T}\left(
\min_{m\in\mathcal M_N}\left\{
\|p_{\mathcal S_m(\mu)}^{\perp}(\mu^{1/2}J_0) -\mu^{1/2}J_0\|_{2}^{2} +
(2 +\mathfrak c_{\rm cal})\frac{m}{N}\right\} +
\frac{1}{N}\right).
\end{displaymath}
\end{theorem}
\noindent
The proof of Theorem \ref{risk_bound_adaptive_estimator} relies on the following lemma, which is a straightforward consequence of a Bernstein type inequality for continuous local martingales vanishing at $0$ (see Revuz and Yor \cite{RY99}, Chapter IV, Exercice 3.16).
%

% Lemma : Bernstein-type inequality.

%
\begin{lemma}\label{Bernstein_type_inequality}
For every $\varepsilon > 0$ and every $\varphi\in\mathbb L^2([0,T],d\langle M\rangle_t)$,
\begin{displaymath}
\mathbb P\left[\frac{1}{N}
\sum_{i = 1}^{N}\int_{0}^{T}\varphi(s)dM_{s}^{i}\geqslant\varepsilon
\right]
\leqslant\exp\left(-\frac{N\varepsilon^2}{2\|\varphi\|_{\langle M\rangle}^{2}}\right).
\end{displaymath}
\end{lemma}
\noindent
Let us establish Theorem \ref{risk_bound_adaptive_estimator}.
%

% Proof.

%
\begin{proof}{\it of Theorem \ref{risk_bound_adaptive_estimator}.}
Let us proceed in three steps.
\\
\\
{\bf Step 1.} As established in the proof of Proposition \ref{risk_bound}, for every $J,K\in\mathcal S_m$,
\begin{displaymath}
\gamma_N(J) -\gamma_N(K) =
\|J - J_0\|_{\langle M\rangle}^{2} -\|K - J_0\|_{\langle M\rangle}^{2} -
\frac{2}{N}\sum_{i = 1}^{N}\int_{0}^{T}(J(s) - K(s))dM_{s}^{i}.
\end{displaymath}
Moreover,
\begin{displaymath}
\gamma_N(\widehat J_{\widehat m,N}) + {\rm pen}(\widehat m)
\leqslant
\gamma_N(\widehat J_{m,N}) + {\rm pen}(m)
\end{displaymath}
for any $m\in\mathcal M_N$, and then
\begin{displaymath}
\gamma_N(\widehat J_{\widehat m,N}) -\gamma_N(\widehat J_{m,N})
\leqslant
{\rm pen}(m) - {\rm pen}(\widehat m).
\end{displaymath}
So, since $\mathcal S_m +\mathcal S_{\widehat m}\subset\mathcal S_{m\vee\widehat m}$ under Assumption \ref{assumption_nested}, and since $2ab\leqslant a^2 + b^2$ for every $a,b\in\mathbb R$,
\begin{eqnarray*}
 \|\widehat J_{\widehat m,N} - J_0\|_{\langle M\rangle}^{2}
 & \leqslant &
 \|\widehat J_{m,N} - J_0\|_{\langle M\rangle}^{2}
 +\frac{2}{N}\sum_{i = 1}^{N}\int_{0}^{T}(\widehat J_{\widehat m,N}(s) -
 \widehat J_{m,N}(s))dM_{s}^{i} + {\rm pen}(m) - {\rm pen}(\widehat m)\\
 & \leqslant &
 \|\widehat J_{m,N} - J_0\|_{\langle M\rangle}^{2}
 + 2\cdot\frac{1}{2}\|\widehat J_{\widehat m,N} -\widehat J_{m,N}\|_{\langle M\rangle}
 \cdot 2\sup_{\varphi\in\mathcal B_{m,\widehat m}}|\nu_N(\varphi)|
 + {\rm pen}(m) - {\rm pen}(\widehat m)\\
 & \leqslant &
 \|\widehat J_{m,N} - J_0\|_{\langle M\rangle}^{2} +
 \frac{1}{4}\|\widehat J_{\widehat m,N} -\widehat J_{m,N}\|_{\langle M\rangle}^{2}\\
 & &
 \hspace{2cm} +
 4\left(\left[\sup_{\varphi\in\mathcal B_{m,\widehat m}}|\nu_N(\varphi)|\right]^2
 - p(m,\widehat m)\right)_+ +
 {\rm pen}(m) + 4p(m,\widehat m) - {\rm pen}(\widehat m),
\end{eqnarray*}
where
\begin{displaymath}
\mathcal B_{m,m'} :=
\{\varphi\in\mathcal S_{m\vee m'} :\|\varphi\|_{\langle M\rangle} = 1\}
\quad {\rm and}\quad
p(m,m') :=
\frac{\mathfrak c_{\rm cal}}{4}\cdot\frac{m\vee m'}{N}
\end{displaymath}
for every $m'\in\mathcal M_N$, and
\begin{displaymath}
\nu_N(\varphi) :=
\frac{1}{N}\sum_{i = 1}^{N}
\int_{0}^{T}\varphi(s)dM_{s}^{i}
\end{displaymath}
for every $\varphi\in\mathbb L^2([0,T],d\langle M\rangle_t)$. Therefore, since $(a + b)^2\leqslant 2a^2 + 2b^2$ for every $a,b\in\mathbb R$, and since $4p(m,\widehat m)\leqslant {\rm pen}(m) + {\rm pen}(\widehat m)$,
\begin{equation}\label{risk_bound_adaptive_estimator_1}
\|\widehat J_{\widehat m,N} - J_0\|_{\langle M\rangle}^{2}
\leqslant
3\|\widehat J_{m,N} - J_0\|_{\langle M\rangle}^{2} + 4{\rm pen}(m) +
8\left(\left[\sup_{\varphi\in\mathcal B_{m,\widehat m}}|\nu_N(\varphi)|\right]^2
- p(m,\widehat m)\right)_+.
\end{equation}
{\bf Step 2.} By using Lemma \ref{Bernstein_type_inequality}, and by following the pattern of the proof of Baraud et al. \cite{BCV07}, Proposition 6.1, the purpose of this step is to find a suitable bound on
\begin{displaymath}
\mathbb E\left[\left(
\left[\sup_{\varphi\in\mathcal B_{m,m'}}|\nu_N(\varphi)|\right]^2
- p(m,m')\right)_+\right]
\textrm{$;$ }
m'\in\mathcal M_N.
\end{displaymath}
Consider $\delta_0\in (0,1)$ and let $(\delta_n)_{n\in\mathbb N^*}$ be the real sequence defined by
\begin{displaymath}
\delta_n :=\delta_02^{-n}
\textrm{$;$ }
\forall n\in\mathbb N^*.
\end{displaymath}
Since $\mathcal S_{m\vee m'}$ is a vector subspace of $\mathbb L^2([0,T],d\langle M\rangle_t)$ of dimension $m\vee m'$, for any $n\in\mathbb N$, by Lorentz et al. \cite{LGM96}, Chapter 15, Proposition 1.3, there exists $T_n\subset\mathcal B_{m,m'}$ such that $|T_n|\leqslant (3/\delta_n)^{m\vee m'}$ and, for any $\varphi\in\mathcal B_{m,m'}$,
\begin{displaymath}
\exists f_n\in T_n :
\|\varphi - f_n\|_{\langle M\rangle}\leqslant\delta_n.
\end{displaymath}
In particular, note that
\begin{displaymath}
\varphi =
f_0 +\sum_{n = 1}^{\infty}(f_n - f_{n - 1}).
\end{displaymath}
Then, for any sequence $(\Delta_n)_{n\in\mathbb N}$ of elements of $(0,\infty)$ such that $\Delta :=\sum_{n\in\mathbb N}\Delta_n <\infty$,
\begin{eqnarray*}
 & & \mathbb P\left[\left(
 \sup_{\varphi\in\mathcal B_{m,m'}}
 |\nu_N(\varphi)|\right)^2 >
 \Delta^2\right]\\
 & &
 \hspace{2cm} =
 \mathbb P\left[\exists (f_n)_{n\in\mathbb N}\in\prod_{n = 0}^{\infty}T_n :
 |\nu_N(f_0)| +\sum_{n = 1}^{\infty}
 |\nu_N(f_n - f_{n - 1})| >\Delta\right]\\
 & &
 \hspace{2cm}\leqslant
 \mathbb P\left[\exists (f_n)_{n\in\mathbb N}\in\prod_{n = 0}^{\infty}T_n :
 |\nu_N(f_0)| >\Delta_0\textrm{ or }[\exists n\in\mathbb N^* :
 |\nu_N(f_n - f_{n - 1})| >\Delta_n]\right]\\
 & &
 \hspace{2cm}\leqslant
 \sum_{f_0\in T_0}\mathbb P[|\nu_N(f_0)| >\Delta_0] +
 \sum_{n = 1}^{\infty}\sum_{(f_{n - 1},f_n)\in\mathbb T_n}
 \mathbb P[|\nu_N(f_n - f_{n - 1})| >\Delta_n]
\end{eqnarray*}
with $\mathbb T_n = T_{n - 1}\times T_n$ for every $n\in\mathbb N^*$. Moreover, $\|f_0\|_{\langle M\rangle}^{2}\leqslant 1$ and
\begin{displaymath}
\|f_n - f_{n - 1}\|_{\langle M\rangle}^{2}
\leqslant
2\delta_{n - 1}^{2} + 2\delta_{n}^{2} =\frac{5}{2}\delta_{n - 1}^{2}
\end{displaymath}
for every $n\in\mathbb N^*$. So, by Lemma \ref{Bernstein_type_inequality},
\begin{eqnarray}
 \mathbb P\left[\left(
 \sup_{\varphi\in\mathcal B_{m,m'}}
 |\nu_N(\varphi)|\right)^2 >\Delta^2\right]
 & \leqslant &
 2\sum_{f_0\in T_0}\exp\left(
 -\frac{N\Delta_{0}^{2}}{2\|f_0\|_{\langle M\rangle}^{2}}\right)
 \nonumber\\
 & &
 \hspace{2cm} +
 2\sum_{n = 1}^{\infty}\sum_{(f_{n - 1},f_n)\in\mathbb T_n}
 \exp\left(
 -\frac{N\Delta_{n}^{2}}{2\|f_n - f_{n - 1}\|_{\langle M\rangle}^{2}}\right)
 \nonumber\\
 \label{risk_bound_adaptive_estimator_2}
 & \leqslant &
 2\exp\left(h_0 -\frac{N\Delta_{0}^{2}}{2}\right) +
 2\sum_{n = 1}^{\infty}\exp\left(h_{n - 1} + h_n -
 \frac{N\Delta_{n}^{2}}{5\delta_{n - 1}^{2}}\right)
\end{eqnarray}
with $h_n =\log(|T_n|)$ for every $n\in\mathbb N$. Now, let us take $\Delta_0$ such that
\begin{displaymath}
h_0 -\frac{N\Delta_{0}^{2}}{2} = -(m\vee m' + x)
\quad {\rm with}\quad x > 0,
\end{displaymath}
which leads to
\begin{displaymath}
\Delta_0 =
\left[\frac{2}{N}(m\vee m' + x + h_0)\right]^{1/2}\leqslant
\left[\frac{2}{N}(1 + h_0)(m\vee m' + x)\right]^{1/2},
\end{displaymath}
and for every $n\in\mathbb N^*$, let us take $\Delta_n$ such that
\begin{displaymath}
h_{n - 1} + h_n -\frac{N\Delta_{n}^{2}}{5\delta_{n - 1}^{2}} =
-(m\vee m' + x + n),
\end{displaymath}
which leads to
\begin{eqnarray*}
 \Delta_n & = &
 \left[\frac{5\delta_{n - 1}^{2}}{N}(m\vee m' + x + h_{n - 1} + h_n + n)\right]^{1/2}\\
 & \leqslant &
 \left[\frac{5\delta_{n - 1}^{2}}{N}(1 +\lambda_n + n)(m\vee m' + x)\right]^{1/2}
 \quad {\rm with}\quad
 \lambda_n = 2\log\left(\frac{3}{\delta_0}\right) + (2n - 1)\log(2).
\end{eqnarray*}
For this appropriate sequence $(\Delta_n)_{n\in\mathbb N}$,
\begin{displaymath}
\mathbb P\left[\left(
\sup_{\varphi\in\mathcal B_{m,m'}}
|\nu_N(\varphi)|\right)^2 >\Delta^2\right]
\leqslant
2e^{-x}e^{-(m\vee m')}\left(1 +\sum_{n = 1}^{\infty}e^{-n}\right)
\leqslant
3.2e^{-x}e^{-(m\vee m')}
\end{displaymath}
by Inequality \eqref{risk_bound_adaptive_estimator_2}, and
\begin{eqnarray*}
 \Delta^2 & \leqslant &
 \frac{1}{N}
 \left[\sqrt 2(1 + h_0)^{1/2}(m\vee m' + x)^{1/2} +
 \sqrt 5\sum_{n = 1}^{\infty}
 \delta_{n - 1}
 (1 +\lambda_n + n)^{1/2}(m\vee m' + x)^{1/2}\right]^2\\
 & \leqslant &
 \frac{\delta^2}{N}(m\vee m' + x)
\end{eqnarray*}
with
\begin{displaymath}
\delta =
\sqrt 2(1 + h_0)^{1/2} +
\sqrt 5\sum_{n = 1}^{\infty}
\delta_{n - 1}(1 +\lambda_n + n)^{1/2} <\infty.
\end{displaymath}
So,
\begin{displaymath}
\mathbb P\left[\left(
\sup_{\varphi\in\mathcal B_{m,m'}}
|\nu_N(\varphi)|\right)^2 -
\frac{\delta^2}{\mathfrak c_{\rm cal}}p(m,m') >
\frac{\delta^2}{N}x\right]
\leqslant
3.2e^{-x}e^{-(m\vee m')}
\end{displaymath}
and then, by taking $\mathfrak c_{\rm cal} > 4\delta^2$ and $y =\delta^2x/N$,
\begin{displaymath}
\mathbb P\left[\left(
\sup_{\varphi\in\mathcal B_{m,m'}}
|\nu_N(\varphi)|\right)^2 - p(m,m') > y\right]
\leqslant
3.2e^{-Ny/\delta^2}e^{-(m\vee m')}.
\end{displaymath}
Therefore,
\begin{eqnarray}
 \mathbb E\left[\left(\left[
 \sup_{\varphi\in\mathcal B_{m,m'}}
 |\nu_N(\varphi)|\right]^2 -
 p(m,m')\right)_+\right]
 & = &
 \int_{0}^{\infty}\mathbb P\left[\left(
 \sup_{\varphi\in\mathcal B_{m,m'}}
 |\nu_N(\varphi)|\right)^2 - p(m,m') > y\right]dy
 \nonumber\\
 \label{risk_bound_adaptive_estimator_3}
 & \leqslant &
 3.2\delta^2\frac{e^{-(m\vee m')}}{N}.
\end{eqnarray}
{\bf Step 3.} By Inequality \eqref{risk_bound_adaptive_estimator_3}, there exists a deterministic constant $\mathfrak c_1 > 0$, not depending on $m$ and $N$, such that
\begin{eqnarray*}
 \mathbb E\left[\left(\left[
 \sup_{\varphi\in\mathcal B_{m,\widehat m}}|\nu_N(\varphi)|\right]^2
 - p(m,\widehat m)\right)_+\right]
 & \leqslant &
 \sum_{m'\in\mathcal M_N}
 \mathbb E\left[\left(\left[
 \sup_{\varphi\in\mathcal B_{m,m'}}
 |\nu_N(\varphi)|\right]^2 -
 p(m,m')\right)_+\right]\\
 & \leqslant &
 \frac{3.2\delta^2}{N}\sum_{m'\in\mathcal M_N}e^{-(m\vee m')}\leqslant
 \frac{3.2\delta^2}{N}\left(me^{-m} +\sum_{m' > m}e^{-m'}\right)
 \leqslant\frac{\mathfrak c_1}{N}.
\end{eqnarray*}
Therefore, by Inequality \eqref{risk_bound_adaptive_estimator_1},
\begin{displaymath}
\mathbb E[\|\widehat J_{\widehat m,N} - J_0\|_{\langle M\rangle}^{2}]
\leqslant
\min_{m\in\mathcal M_N}\{
3\mathbb E[\|\widehat J_{m,N} - J_0\|_{\langle M\rangle}^{2}] + 4{\rm pen}(m)\} +
\frac{8\mathfrak c_1}{N}.
\quad\qed
\end{displaymath}
\end{proof}
\noindent
As in the usual nonparametric regression framework, since ${\rm pen}(m)$ is of same order than the bound on the variance term of $\widehat J_{m,N}$ for every $m\in\mathcal M_N$, Theorem \ref{risk_bound_adaptive_estimator} says that the risk of our adaptive estimator is controlled by the minimal risk of $\widehat J_{.,N}$ on $\mathcal M_N$ up to a multiplicative constant not depending on $N$.
\\
\\
Finally, the following proposition provides a risk bound on the discrete-time estimator $\widehat J_{m,N,n}$.
%

% Proposition : Risk bound on the discrete-time estimator.

%
\begin{proposition}\label{risk_bound_discrete}
Under Assumption \ref{assumption_quadratic_variation_strong}, there exists a deterministic constant $\mathfrak c_{\ref{risk_bound_discrete}} > 0$, not depending on $m$, $N$ and $n$, such that
\begin{displaymath}
\mathbb E[\|\widehat J_{m,N,n} - J_0\|_{\langle M\rangle}^{2}]
\leqslant
2\min_{J\in\mathcal S_m}\|J - J_0\|_{\langle M\rangle}^{2} +
\mathfrak c_{\ref{risk_bound_discrete}}\left(
\frac{m}{N} +\frac{mR(m)}{n^2}\right),
\end{displaymath}
where
\begin{displaymath}
R(m) :=\sup_{t\in [0,T]}\sum_{j = 1}^{m}\varphi_j'(t)^2.
\end{displaymath}
\end{proposition}
\noindent
The proof of Proposition \ref{risk_bound_discrete} relies on the following technical lemma.
%

% Lemma : Technical lemma for the risk bound on the discrete-time estimator.

%
\begin{lemma}\label{technical_lemma_discrete}
Under Assumptions \ref{assumption_quadratic_variation_strong},
\begin{displaymath}
\int_{0}^{T}\|\Psi_{m}^{-1}\varphi(t)\|_{2,m}^{2}d\langle M\rangle_t =
{\rm trace}(\Psi_{m}^{-1})\leqslant\|\mu(.)^{-1}\|_{\infty,T}m
\end{displaymath}
with $\varphi = (\varphi_1,\dots,\varphi_m)$.
\end{lemma}
%

% Proof.

%
\begin{proof}
Since $\Psi_m$ is a symmetric matrix,
\begin{eqnarray*}
 \int_{0}^{T}\|\Psi_{m}^{-1}\varphi(t)\|_{2,m}^{2}d\langle M\rangle_t
 & = &
 \sum_{j = 1}^{m}\int_{0}^{T}
 \left(\sum_{k = 1}^{m}[\Psi_{m}^{-1}]_{j,k}\varphi_k(t)\right)^2d\langle M\rangle_t\\
 & = &
 \sum_{j,k,k' = 1}^{m}[\Psi_{m}^{-1}]_{j,k}[\Psi_{m}^{-1}]_{j,k'}
 \int_{0}^{T}\varphi_k(t)\varphi_{k'}(t)d\langle M\rangle_t\\
 & = &
 \sum_{j,k = 1}^{m}[\Psi_{m}^{-1}]_{j,k}
 \sum_{k' = 1}^{m}[\Psi_{m}^{-1}]_{j,k'}[\Psi_m]_{k',k}\\
 & = &
 \sum_{j,k = 1}^{m}[\Psi_{m}^{-1}]_{j,k}I_{j,k} =
 {\rm trace}(\Psi_{m}^{-1})
\end{eqnarray*}
and
\begin{eqnarray*}
 \|\Psi_{m}^{-1}\|_{\rm op} & = &
 \sup_{\|\theta\|_{2,m} = 1}\theta^*\Psi_{m}^{-1}\theta =
 \sup_{\|\theta\|_{2,m} = 1}\|\Psi_{m}^{-1/2}\theta\|_{2,m}^{2} =
 \sup_{\|\Psi_{m}^{1/2}\theta\|_{2,m} = 1}\|\theta\|_{2,m}^{2}\\
 & = &
 \sup_{J\in\mathcal S_m :\|J\|_{\langle M\rangle} = 1}\|J\|_{2}^{2} =
 \sup_{J\in\mathcal S_m :\|J\|_{\langle M\rangle} = 1}
 \int_{0}^{T}J(s)^2\mu(s)^{-1}d\langle M\rangle_s
 \leqslant\|\mu(.)^{-1}\|_{\infty,T}.
\end{eqnarray*}
Therefore,
\begin{displaymath}
\int_{0}^{T}\|\Psi_{m}^{-1}\varphi(t)\|_{2,m}^{2}d\langle M\rangle_t =
{\rm trace}(\Psi_{m}^{-1})
\leqslant m\|\Psi_{m}^{-1}\|_{\rm op}
\leqslant m\|\mu(.)^{-1}\|_{\infty,T}.
\quad\qed
\end{displaymath}
\end{proof}
\noindent
Let us establish Proposition \ref{risk_bound_discrete}.
%

% Proof.

%
\begin{proof}{\it of Proposition \ref{risk_bound_discrete}.}
First of all, note that
\begin{eqnarray*}
 \mathbb E[\|\widehat J_{m,N,n} - J_0\|_{\langle M\rangle}^{2}]
 & \leqslant &
 2\mathbb E[\|\widehat J_{m,N} - J_0\|_{\langle M\rangle}^{2}] +
 2\mathbb E[\|\widehat J_{m,N} -\widehat J_{m,N,n}\|_{\langle M\rangle}^{2}]\\
 & \leqslant &
 2\left(\min_{J\in\mathcal S_m}\|J - J_0\|_{\langle M\rangle}^{2} +
 \frac{2m}{N} +\Delta_{m,N,n}\right),
\end{eqnarray*}
where
\begin{displaymath}
\Delta_{m,N,n} :=
\int_{0}^{T}\mathbb E[\langle\Psi_{m}^{-1}(z_{m,N} - z_{m,N,n}),\varphi(t)\rangle_{2,m}^{2}]d\langle M\rangle_t.
\end{displaymath}
Since $Z^1,\dots,Z^N$ are independent copies of $Z$, and since $\Psi_{m}^{-1}$ is a symmetric matrix,
\begin{eqnarray*}
 \Delta_{m,N,n} & = &
 \frac{1}{N^2}\int_{0}^{T}\mathbb E\left[\left|
 \sum_{i = 1}^{N}\sum_{l = 0}^{n - 1}\int_{t_l}^{t_{l + 1}}
 \langle\Psi_{m}^{-1}(\varphi(s) -\varphi(t_l)),
 \varphi(t)\rangle_{2,m}dZ_{s}^{i}\right|^2\right]d\langle M\rangle_t\\
 & \leqslant &
 2\int_{0}^{T}\left(
 \sum_{l = 0}^{n - 1}
 \int_{t_l}^{t_{l + 1}}\langle\varphi(s) -\varphi(t_l),
 \Psi_{m}^{-1}\varphi(t)\rangle_{2,m}J_0(s)
 d\langle M\rangle_s\right)^2d\langle M\rangle_t\\
 & &
 \hspace{1cm}
 + 2\int_{0}^{T}\mathbb E\left[\left|
 \sum_{l = 0}^{n - 1}
 \int_{t_l}^{t_{l + 1}}\langle\varphi(s) -\varphi(t_l),
 \Psi_{m}^{-1}\varphi(t)\rangle_{2,m}dM_s\right|^2\right]d\langle M\rangle_t\\
 & =: &
 2A_{m,n} + 2B_{m,n}.
\end{eqnarray*}
Now, let us find suitable bounds on $A_{m,n}$ and $B_{m,n}$:
\begin{itemize}
 \item {\bf Bound on $A_{m,n}$.} By Cauchy-Schwarz's inequality and Lemma \ref{technical_lemma_discrete},
 \begin{eqnarray*}
  A_{m,n} & \leqslant &
  \left(
  \int_{0}^{T}\|\Psi_{m}^{-1}\varphi(t)\|_{2,m}^{2}d\langle M\rangle_t\right)
  \left(\sum_{l = 0}^{n - 1}\int_{t_l}^{t_{l + 1}}
  \|\varphi(s) -\varphi(t_l)\|_{2,m}|J_0(s)|
  d\langle M\rangle_s\right)^2\\
  & \leqslant &
  {\rm trace}(\Psi_{m}^{-1})\|\varphi'\|_{\infty,T}^{2}
  \left(\sum_{l = 0}^{n - 1}\int_{t_l}^{t_{l + 1}}(s - t_l)|J_0(s)|d\langle M\rangle_s\right)^2\\
  & \leqslant &
  {\rm trace}(\Psi_{m}^{-1})R(m)
  \frac{T^2}{n^2}
  \left(\sum_{l = 0}^{n - 1}\int_{t_l}^{t_{l + 1}}|J_0(s)|d\langle M\rangle_s\right)^2 =
  \|\mu(.)^{-1}\|_{\infty,T}\|J_0\|_{\langle M\rangle}^{2}\langle M\rangle_TT^2
  \frac{mR(m)}{n^2}.
 \end{eqnarray*}
 \item {\bf Bound on $B_{m,n}$.} By the isometry property of It\^o's integral, Cauchy-Schwarz's inequality and Lemma \ref{technical_lemma_discrete},
 \begin{eqnarray*}
  B_{m,n} & = &
  \int_{0}^{T}\mathbb E\left[\left|\int_{0}^{T}\left(
  \sum_{l = 0}^{n - 1}\langle\varphi(s) -\varphi(t_l),
  \Psi_{m}^{-1}\varphi(t)\rangle_{2,m}\mathbf 1_{[t_l,t_{l + 1}]}(s)
  \right)dM_s\right|^2\right]d\langle M\rangle_t\\
  & = &
  \int_{0}^{T}\int_{0}^{T}\left(
  \sum_{l = 0}^{n - 1}\langle\varphi(s) -\varphi(t_l),
  \Psi_{m}^{-1}\varphi(t)\rangle_{2,m}^{2}\mathbf 1_{[t_l,t_{l + 1}]}(s)
  \right)d\langle M\rangle_sd\langle M\rangle_t\\
  & \leqslant &
  \left(\int_{0}^{T}\|\Psi_{m}^{-1}\varphi(t)\|_{2,m}^{2}d\langle M\rangle_t\right)
  \left(
  \sum_{l = 0}^{n - 1}\int_{t_l}^{t_{l + 1}}
  \|\varphi(s) -\varphi(t_l)\|_{2,m}^{2}d\langle M\rangle_s\right)\\
  & \leqslant &
  {\rm trace}(\Psi_{m}^{-1})
  \|\varphi'\|_{\infty,T}^{2}\sum_{l = 0}^{n - 1}
  \int_{t_l}^{t_{l + 1}}(s - t_l)^2d\langle M\rangle_s\\
  & \leqslant &
  {\rm trace}(\Psi_{m}^{-1})R(m)
  \frac{T^2}{n^2}
  \sum_{l = 0}^{n - 1}\int_{t_l}^{t_{l + 1}}d\langle M\rangle_s
  \leqslant
  \|\mu(.)^{-1}\|_{\infty,T}\langle M\rangle_TT^2\frac{mR(m)}{n^2}.
 \end{eqnarray*}
\end{itemize}
In conclusion, there exists a deterministic constant $\mathfrak c_1 > 0$, not depending on $m$, $N$ and $n$, such that
\begin{displaymath}
\mathbb E[\|\widehat J_{m,N,n} - J_0\|_{\langle M\rangle}^{2}]
\leqslant
2\min_{J\in\mathcal S_m}\|J - J_0\|_{\langle M\rangle}^{2} +
\mathfrak c_1\left(
\frac{m}{N} +\frac{mR(m)}{n^2}\right).
\quad\qed
\end{displaymath}
\end{proof}
\noindent
For instance, assume that $(\varphi_1,\dots,\varphi_m)$ is the trigonometric basis. As established in Comte and Marie \cite{CM21}, Subsubsection 3.2.1, $R(m)$ is of order $m^3$. Then, for $n$ of order $N^2$, the variance term in the risk bound of Proposition \ref{risk_bound_discrete} is of order $m/N$ as in the risk bound on the continuous-time estimator of Proposition \ref{risk_bound}.
%

% Section : Application to differential equations driven by the fractional Brownian motion.

%
\section{Application to differential equations driven by the fractional Brownian motion}\label{section_application_fDE}
Throughout this section, $M$ is the Molchan martingale defined at Section \ref{projection_LS_estimator_section}. For $H = 1/2$, we assume that $V :\mathbb R\rightarrow\mathbb R$ is continuously differentiable, $V'$ is bounded, $\sigma : [0,T]\rightarrow\mathbb R\backslash\{0\}$ and $b_0 : [0,T]\rightarrow\mathbb R$ are continuous, and then Equation \eqref{main_fDE} has a unique solution (see Revuz and Yor \cite{RY99}, Chapter IX, Theorem 2.1). For $H\in (1/2,1)$, we assume that $V :\mathbb R\rightarrow\mathbb R$ is twice continuously differentiable, $V'$ and $V''$ are bounded, $\sigma : [0,T]\rightarrow\mathbb R\backslash\{0\}$ is $\gamma$-H\"older continuous with $\gamma\in (1 - H,1]$, $b_0 : [0,T]\rightarrow\mathbb R$ is continuous, and then Equation \eqref{main_fDE} has a unique solution which paths are $\alpha$-H\"older continuous from $[0,T]$ into $\mathbb R$ for every $\alpha\in (1/2,H)$ (see Kubilius et al. \cite{KMR17}, Theorem 1.42). In the sequel, the maps $V$ and $\sigma$ are known and our purpose is to provide a nonparametric estimator of $b_0$.
%

% Subsection : Auxiliary model.

%
\subsection{Auxiliary model}\label{subsection_auxiliary_model}
The model transformation used in the sequel has been introduced in Kleptsyna and Le Breton \cite{KL01} in the parametric estimation framework. Let $Q_0 : [0,T]\rightarrow\mathbb R$ be the map defined by
\begin{displaymath}
Q_0(t) :=\frac{b_0(t)}{\sigma(t)}
\textrm{$;$ }
\forall t\in [0,T],
\end{displaymath}
and assume that $Q_0\in\mathcal Q := C^1([0,T];\mathbb R)$. Consider also the process $Z$ such that for every $t\in [0,T]$,
\begin{displaymath}
Z_t :=
\int_{0}^{t}\ell(t,s)dY_s
\end{displaymath}
with
\begin{displaymath}
Y_t =\int_{0}^{t}\frac{dX_s}{V(X_s)\sigma(s)} =
\int_{0}^{t}\left(\frac{b_0(s)}{\sigma(s)}ds + dB_s\right) =
\int_{0}^{t}Q_0(s)ds + B_t.
\end{displaymath}
Then, Equation \eqref{main_fDE} leads to
\begin{eqnarray}
 Z_t & = & j(Q_0)(t) + M_t
 \nonumber\\
 \label{main_fDE_rewritten}
 & = &
 \int_{0}^{t}J(Q_0)(s)d\langle M\rangle_s + M_t,
\end{eqnarray}
where
\begin{displaymath}
j(Q)(t) :=
\int_{0}^{t}\ell(t,s)Q(s)ds
\quad {\rm and}\quad
J(Q)(t) :=
(2 - 2H)^{-1}t^{2H - 1}j(Q)'(t)
\end{displaymath}
for any $Q\in\mathcal Q$ and every $t\in (0,T]$. About the existence of $j(Q)'(t)$ for every $t\in (0,T]$, see Kubilius et al. \cite{KMR17}, Lemma 5.8.
%

% Subsection : An estimator of Q_0.

%
\subsection{An estimator of $Q_0$}\label{subsection_estimator_Q_0}
For $H = 1/2$, $\langle M\rangle_t =\langle B\rangle_t = t$ for every $t\in [0,T]$, and $J(Q) = Q$ for every $Q\in\mathcal Q$. Then, in Model \eqref{main_fDE_rewritten}, for any $m\in\{1,\dots,N\}$, the solution to Problem \eqref{estimator} is a nonparametric estimator of $Q_0$. So, no additional investigations are required when $H = 1/2$. For $H\in (1/2,1)$, in Model \eqref{main_fDE_rewritten}, the solution $\widehat J_{m,N}$ to Problem \eqref{estimator} is a nonparametric estimator of $J(Q_0)$. So, for $H\in (1/2,1)$, this subsection deals with an estimator of $Q_0$ defined from $\widehat J_{m,N}$.
\\
\\
Let us consider the function space
\begin{displaymath}
\mathcal J :=
\left\{\iota :\textrm{the function }
t\in [0,T]\longmapsto\int_{0}^{t}s^{1 - 2H}\iota(s)ds
\textrm{ belongs to }\mathcal I_{0+}^{3/2 - H}(\mathbb L^1([0,T],dt))
\right\},
\end{displaymath}
where $\mathcal I_{0+}^{3/2 - H}(.)$ is the Riemann-Liouville left-sided fractional integral of order $3/2 - H$. The reader can refer to Samko et al. \cite{SKM93}, Chapter 1, Section 2 on fractional calculus.
\\
\\
In order to define our estimator of $Q_0$, let us establish first the following technical proposition.
%

% Proposition : Inverse of the map J.

%
\begin{proposition}\label{inverse_J}
The map $J : Q\mapsto J(Q)$ satisfies the two following properties:
\begin{enumerate}
 \item $J(\mathcal Q)\subset\mathcal J$.
 \item $\overline J(J(Q)) = Q$ for every $Q\in\mathcal Q$, where $\overline J$ is the map defined on $\mathcal J$ by
 \begin{displaymath}
 \overline J(\iota)(t) :=
 \overline{\mathfrak c}_Ht^{H - 1/2}
 \int_{0}^{t}(t - s)^{H - 3/2}s^{1 - 2H}\iota(s)ds
 \end{displaymath}
 for every $\iota\in\mathcal J$ and $t\in [0,T]$, where
 \begin{displaymath}
 \overline{\mathfrak c}_H :=
 \frac{2 - 2H}{\mathfrak c_H\Gamma(3/2 - H)\Gamma(H - 1/2)}.
 \end{displaymath}
\end{enumerate}
\end{proposition}
%

% Proof.

%
\begin{proof}
In the sequel, $\mathcal I_{0+}^{\alpha}(.)$ (resp. $\mathcal D_{0+}^{\alpha}(.)$) is the Riemann-Liouville left-sided fractional integral (resp. derivative) of order $\alpha\in (0,1)$. Consider $Q\in\mathcal Q$ and the function $Q_H : (0,T]\rightarrow\mathbb R$ defined by
\begin{displaymath}
Q_H(t) :=
t^{1/2 - H}Q(t)
\textrm{ $;$ }
\forall t\in (0,T].
\end{displaymath}
The function $\mathcal I_{0+}^{3/2 - H}(Q_H)$ is well-defined on $(0,T]$ and, for every $t\in (0,T]$,
\begin{eqnarray*}
 \mathcal I_{0+}^{3/2 - H}(Q_H)(t) & = &
 \frac{1}{\Gamma(3/2 - H)}\int_{0}^{t}(t - s)^{1/2 - H}Q_H(s)ds\\
 & = &
 \frac{1}{\mathfrak c_H\Gamma(3/2 - H)}\int_{0}^{t}\ell(t,s)Q(s)ds =
 \frac{1}{\mathfrak c_H\Gamma(3/2 - H)}j(Q)(t).
\end{eqnarray*}
Since $j(Q)'(t)$ exists for any $t\in (0,T]$ by Kubilius et al. \cite{KMR17}, Lemma 5.8 as mentioned above, the derivative of $\mathcal I_{0+}^{3/2 - H}(Q_H)$ at time $t$ also. Moreover, since $Q$ is continuous on $[0,T]$,
\begin{eqnarray*}
 |\mathcal I_{0+}^{3/2 - H}(Q_H)(t)| & \leqslant &
 \frac{1}{\Gamma(3/2 - H)}\int_{0}^{t}s^{1/2 - H}(t - s)^{1/2 - H}|Q(s)|ds\\
 & \leqslant &
 \frac{\|Q\|_{\infty,T}}{2\Gamma(3/2 - H)}
 \left[\int_{0}^{t}s^{1 - 2H}ds +
 \int_{0}^{t}(t - s)^{1 - 2H}ds\right]\\
 & = &
 \frac{\|Q\|_{\infty,T}}{(2 - 2H)\Gamma(3/2 - H)}t^{2 - 2H}
 \xrightarrow[t\rightarrow 0]{}0.
\end{eqnarray*}
By the definition of the map $J$, for every $s\in (0,T]$,
\begin{eqnarray*}
 J(Q)(s) & = &
 (2 - 2H)^{-1}s^{2H - 1}j(Q)'(s)\\
 & = &
 \frac{\mathfrak c_H\Gamma(3/2 - H)}{2 - 2H}
 s^{2H - 1}\frac{\partial}{\partial s}
 \mathcal I_{0+}^{3/2 - H}(Q_H)(s)
\end{eqnarray*}
and then, for every $t\in [0,T]$,
\begin{eqnarray}
 \int_{0}^{t}s^{1 - 2H}J(Q)(s)ds & = &
 \frac{\mathfrak c_H\Gamma(3/2 - H)}{2 - 2H}
 \lim_{\varepsilon\rightarrow 0}\int_{\varepsilon}^{t}
 \left[\frac{\partial}{\partial s}\mathcal I_{0+}^{3/2 - H}(Q_H)(s)\right]ds
 \nonumber\\
 & = &
 \frac{\mathfrak c_H\Gamma(3/2 - H)}{2 - 2H}
 \lim_{\varepsilon\rightarrow 0}[\mathcal I_{0+}^{3/2 - H}(Q_H)(t) -
 \mathcal I_{0+}^{3/2 - H}(Q_H)(\varepsilon)]
 \nonumber\\
 \label{inverse_J_1}
 & = &
 \frac{\mathfrak c_H\Gamma(3/2 - H)}{2 - 2H}
 \mathcal I_{0+}^{3/2 - H}(Q_H)(t).
\end{eqnarray}
So, $J(Q)\in\mathcal J$ by Equality \eqref{inverse_J_1}, and then $J(\mathcal Q)\subset\mathcal J$. By applying the Riemann-Liouville left-sided fractional derivative of order $3/2 - H$ on each side of Equality \eqref{inverse_J_1}, and thanks to its representation for absolutely continuous functions on $[0,T]$ (see Kubilius et al. \cite{KMR17}, Proposition 1.10), for every $t\in (0,T]$,
\begin{eqnarray*}
 Q_H(t) & = &
 \frac{2 - 2H}{\mathfrak c_H\Gamma(3/2 - H)}
 \mathcal D_{0+}^{3/2 - H}\left(\int_{0}^{.}
 s^{1 - 2H}J(Q)(s)ds\right)(t)\\
 & = &
 \overline{\mathfrak c}_H
 \int_{0}^{t}(t - s)^{H - 3/2}s^{1 - 2H}J(Q)(s)ds.
\end{eqnarray*}
Therefore, for every $t\in [0,T]$,
\begin{displaymath}
Q(t) =
\overline{\mathfrak c}_Ht^{H - 1/2}
\int_{0}^{t}(t - s)^{H - 3/2}s^{1 - 2H}J(Q)(s)ds =
\overline J(J(Q))(t).
\quad\qed
\end{displaymath}
\end{proof}
\noindent
For $\varphi_1,\dots,\varphi_m\in\mathcal J$, Proposition \ref{inverse_J} legitimates to consider the following estimator of $Q_0$:
\begin{eqnarray*}
 \widehat Q_{m,N}(t) & := &
 \overline J(\widehat J_{m,N})(t)\\
 & = &
 \overline{\mathfrak c}_Ht^{H - 1/2}
 \int_{0}^{t}(t - s)^{H - 3/2}s^{1 - 2H}\widehat J_{m,N}(s)ds
 \textrm{$;$ }
 t\in [0,T].
\end{eqnarray*}
The following proposition provides risk bounds on $\widehat Q_{m,N}$, $m\in\{1,\dots,N\}$, and on the adaptive estimator $\widehat Q_{\widehat m,N}$.
%

% Proposition : Risk bound on the estimator of Q_0.

%
\begin{proposition}\label{risk_bound_estimator_Q_0}
Assume that $J(Q_0)\in\mathbb L^2([0,T],d\langle M\rangle_t)$. If the $\varphi_j$'s belong to $\mathcal J$, then there exists a deterministic constant $\mathfrak c_{\ref{risk_bound_estimator_Q_0},1} > 0$, not depending on $N$, such that
\begin{displaymath}
\mathbb E[\|\widehat Q_{m,N} - Q_0\|_{2}^{2}]
\leqslant
\mathfrak c_{\ref{risk_bound_estimator_Q_0},1}
\left(\min_{\iota\in\mathcal S_m}
\|\iota - J(Q_0)\|_{\langle M\rangle}^{2} +\frac{m}{N}\right)
\textrm{$;$ }\forall m\in\{1,\dots,N\}.
\end{displaymath}
If in addition the $\varphi_j$'s fulfill Assumption \ref{assumption_nested}, then there exists a deterministic constant $\mathfrak c_{\ref{risk_bound_estimator_Q_0},2} > 0$, not depending on $N$, such that
\begin{displaymath}
\mathbb E[\|\widehat Q_{\widehat m,N} - Q_0\|_{2}^{2}]
\leqslant
\mathfrak c_{\ref{risk_bound_estimator_Q_0},2}\left(
\min_{m\in\mathcal M_N}\left\{\min_{\iota\in\mathcal S_m}
\|\iota - J(Q_0)\|_{\langle M\rangle}^{2}
+\frac{m}{N}\right\} +\frac{1}{N}\right).
\end{displaymath}
\end{proposition}
%

% Proof.

%
\begin{proof}
Consider $\iota\in\mathcal J\cap\mathbb L^2([0,T],d\langle M\rangle_t)$. By Cauchy-Schwarz's inequality,
\begin{eqnarray*}
 \|\overline J(\iota)\|_{2}^{2} & = &
 \overline{\mathfrak c}_{H}^{2}
 \int_{0}^{T}t^{2H - 1}\left(
 \int_{0}^{t}(t - s)^{H - 3/2}s^{1 - 2H}\iota(s)ds
 \right)^2dt\\
 & \leqslant &
 \overline{\mathfrak c}_{H}^{2}
 \int_{0}^{T}t^{2H - 1}\theta(t)
 \int_{0}^{t}(t - s)^{H - 3/2}s^{1 - 2H}\iota(s)^2dsdt
\end{eqnarray*}
with
\begin{displaymath}
\theta(t) :=
\int_{0}^{t}(t - s)^{H - 3/2}s^{1 - 2H}ds
\textrm{$;$ }
\forall t\in (0,T].
\end{displaymath}
Note that $\overline\theta : t\mapsto t^{2H - 1}\theta(t)$ is bounded on $(0,T]$. Indeed, for every $t\in (0,T]$,
\begin{eqnarray*}
 |\overline\theta(t)| & \leqslant &
 t^{2H - 1}\left[\left(\frac{t}{2}\right)^{H - 3/2}\int_{0}^{t/2}s^{1 - 2H}ds +
 \left(\frac{t}{2}\right)^{1 - 2H}\int_{t/2}^{t}(t - s)^{H - 3/2}ds
 \right]\\
 & = &
 t^{2H - 1}\left[\frac{1}{2 - 2H}\left(\frac{t}{2}\right)^{H - 3/2}
 \left(\frac{t}{2}\right)^{2 - 2H} +
 \frac{1}{H - 1/2}\left(\frac{t}{2}\right)^{1 - 2H}\left(\frac{t}{2}\right)^{H - 1/2}
 \right]\\
 & = &
 \frac{1}{2^{1/2 - H}}
 \left(\frac{1}{2 - 2H} +\frac{1}{H - 1/2}\right)
 t^{H - 1/2}\xrightarrow[t\rightarrow 0^+]{} 0.
\end{eqnarray*}
So, by Fubini's theorem,
\begin{eqnarray*}
 \|\overline J(\iota)\|_{2}^{2}
 & \leqslant &
 \overline{\mathfrak c}_{H}^{2}\|\overline\theta\|_{\infty,T}
 \int_{0}^{T}\int_{0}^{T}
 (t - s)^{H - 3/2}s^{1 - 2H}\iota(s)^2\mathbf 1_{(s,T]}(t)dsdt\\
 & = &
 \overline{\mathfrak c}_{H}^{2}\|\overline\theta\|_{\infty,T}
 \int_{0}^{T}s^{1 - 2H}\iota(s)^2\int_{s}^{T}
 (t - s)^{H - 3/2}dtds\\
 & \leqslant &
 (H - 1/2)^{-1}\overline{\mathfrak c}_{H}^{2}T^{H - 1/2}\|\overline\theta\|_{\infty,T}
 \int_{0}^{T}s^{1 - 2H}\iota(s)^2ds.
\end{eqnarray*}
Then, $\overline J(\iota)\in\mathbb L^2([0,T],dt)$ and
\begin{equation}\label{risk_bound_estimator_Q_0_1}
\|\overline J(\iota)\|_{2}^{2}
\leqslant
\mathfrak c_1\|\iota\|_{\langle M\rangle}^{2}
\quad {\rm with}\quad
\mathfrak c_1 = (H - 1/2)^{-1}(2 - 2H)^{-1}
\overline{\mathfrak c}_{H}^{2}T^{H - 1/2}\|\overline\theta\|_{\infty,T}.
\end{equation}
Therefore, by the definition of the estimator $\widehat Q_{m,N}$, by Proposition \ref{inverse_J} and by Inequality \eqref{risk_bound_estimator_Q_0_1},
\begin{displaymath}
\|\widehat Q_{m,N} - Q_0\|_{2}^{2} =
\|\overline J(\widehat J_{m,N}) -\overline J(J(Q_0))\|_{2}^{2} =
\|\overline J(\widehat J_{m,N} - J(Q_0))\|_{2}^{2}\leqslant
\mathfrak c_1\|\widehat J_{m,N} - J(Q_0)\|_{\langle M\rangle}^{2}.
\end{displaymath}
The conclusion comes from Proposition \ref{risk_bound} and Theorem \ref{risk_bound_adaptive_estimator}.\qed
\end{proof}
\noindent
Proposition \ref{risk_bound_estimator_Q_0} says that the MISE (Mean Integrated Squared Error) of $\widehat Q_{m,N}$, $m\in\{1,\dots,N\}$, (resp. $\widehat Q_{\widehat m,N}$) has at most the same bound than the MISE of $\widehat J_{m,N}$ (resp. $\widehat J_{\widehat m,N}$).
%

% Subsection : Example 1: drift estimation in a non-autonomous Black-Scholes model.

%
\subsection{Example 1: drift estimation in a non-autonomous Black-Scholes model}\label{subsection_Black_Scholes}
Let us consider a financial market model in which the prices process $S = (S_t)_{t\in\mathbb R_+}$ of the risky asset satisfies
\begin{equation}\label{Black_Scholes_model}
S_t = S_0 +
\int_{0}^{t}S_u(b_0(u)du +\sigma dW_u)
\textrm{$;$ }t\in\mathbb R_+,
\end{equation}
where $S_0$ is a $(0,\infty)$-valued random variable, $W = (W_t)_{t\in\mathbb R_+}$ is a Brownian motion, $\sigma > 0$ and $b_0\in C^0(\mathbb R_+;\mathbb R)$. This is a non-autonomous extension of the Black-Scholes model.
\\
\\
Note that, in practice, several independent copies of the prices process $S$ cannot be observed on $[0,T]$. So, in order to define a suitable estimator of $b_0$ on $[0,T]$, let us assume that $S$ is observed on $[0,NT]$ and, for any $i\in\{1,\dots,N\}$, consider $T_i := (i - 1)T$ and the process $S^i = (S_{t}^{i})_{t\in [0,T]}$ defined by
\begin{displaymath}
S_{t}^{i} := S_{T_i + t}
\textrm{$;$ }\forall t\in [0,T].
\end{displaymath}
By Equation \eqref{Black_Scholes_model}, for every $t\in [0,T]$,
\begin{eqnarray*}
 S_{t}^{i} & = &
 S_{T_i} +\int_{T_i}^{T_i + t}S_u(b_0(u)du +\sigma dW_u)\\
 & = &
 S_{0}^{i} +\int_{0}^{t}S_{u}^{i}(b_0(T_i + u)du +\sigma dW_{u}^{i})
 \quad {\rm with}\quad
 W^i := W_{T_i +\cdot} - W_{T_i}.
\end{eqnarray*}
Moreover, let $Z^i = (Z_{t}^{i})_{t\in [0,T]}$ be the process defined by
\begin{equation}\label{process_Z_Black_Scholes}
Z_{t}^{i}
:=\frac{1}{\sigma}\int_{0}^{t}\frac{dS_{u}^{i}}{S_{u}^{i}}
=\frac{1}{\sigma}\int_{0}^{t}b_0(T_i + u)d\langle W^i\rangle_u + W_{t}^{i}
\end{equation}
for every $t\in [0,T]$. Since $W^1,\dots,W^N$ are $N$ independent Brownian motions, by assuming that the volatility constant $\sigma$ is known and that $b_0$ is $T$-periodic, a suitable nonparametric estimator of $b_0$ on $[0,T]$ is given by
\begin{displaymath}
\widehat b_{m,N}(t) :=
\sigma\widehat J_{m,N}(t)
\textrm{$;$ }
t\in [0,T],
\end{displaymath}
where $m\in\{1,\dots,N\}$ and
\begin{displaymath}
\widehat J_{m,N} =\arg\min_{J\in\mathcal S_m}\left\{
\frac{1}{N}\sum_{i = 1}^{N}\left(
\int_{0}^{T}J(s)^2ds -
2\int_{0}^{T}J(s)dZ_{s}^{i}\right)\right\}.
\end{displaymath}
Since
\begin{displaymath}
\|\widehat b_{m,N} - b_0\|_{2}^{2} =
\sigma^2\|\widehat J_{m,N} - J_0\|_{2}^{2}
\quad {\rm with}\quad
J_0 =\frac{b_0}{\sigma},
\end{displaymath}
Proposition \ref{risk_bound} provides a risk bound on $\widehat b_{m,N}$, and Theorem \ref{risk_bound_adaptive_estimator} provides a risk bound on $\widehat b_{\widehat m,N}$ with $\widehat m$ selected in $\mathcal M_N\subset\{1,\dots,N\}$ via \eqref{model_selection_method}.
\\
\\
Finally, to assume that $b_0$ is $T$-periodic means that Model \eqref{Black_Scholes_model} is appropriate for assets with a prices process having {\it similar} trends on each interval $[T_i,T_{i + 1}]$, typically each day ($T = 24$h). Obviously, since constant functions are $T$-periodic, $\widehat b_{m,N}$ is an estimator of the drift constant in the usual Black-Scholes model.
%

% Subsection : Example 2: nonparametric estimation in a non-autonomous fractional stochastic volatility model.

%
\subsection{Example 2: nonparametric estimation in a non-autonomous fractional stochastic volatility model}\label{subsection_fractional_stochastic_volatility}
Let us consider a financial market model in which the prices process $S = (S_t)_{t\in [0,T]}$ of the risky asset satisfies Equation \eqref{fractional_stochastic_volatility_model}, that is
\begin{displaymath}
\left\{
\begin{array}{rcl}
 dS_t & = & S_t(b(t)dt +\sigma_tdW_t)\\
 d\sigma_t & = & \sigma_t(\rho_0(t)dt +\upsilon dB_t)
\end{array}\right.,
\end{displaymath}
where $S_0$ and $\sigma_0$ are $(0,\infty)$-valued random variables, $W = (W_t)_{t\in\mathbb R_+}$ (resp. $B = (B_t)_{t\in\mathbb R_+}$) is a Brownian motion (resp. a fractional Brownian motion of Hurst parameter $H\in [1/2,1)$), $W$ and $B$ are independent, $\upsilon > 0$ and $b,\rho_0\in C^0(\mathbb R_+;\mathbb R)$.
\\
\\
Here again, in practice, it is not possible to get several independent copies of the volatility process $\sigma$ on $[0,T]$. So, in order to define a suitable estimator of $\rho_0$ on $[0,T]$, let us assume that $(S,\sigma)$ is observed on $[0,N(T +\Delta)]$ with $\Delta\in\mathbb R_+$, and for any $i\in\{1,\dots,N\}$, consider $T_i(\Delta) := (i - 1)(T +\Delta)$ and the process $\sigma^i = (\sigma_{t}^{i})_{t\in [0,T]}$ defined by
\begin{displaymath}
\sigma_{t}^{i} :=\sigma_{T_i(\Delta) + t}
\textrm{$;$ }\forall t\in [0,T].
\end{displaymath}
By Equation \eqref{fractional_stochastic_volatility_model}, for every $t\in [0,T]$,
\begin{eqnarray*}
 \sigma_{t}^{i} & = &
 \sigma_{T_i(\Delta)} +
 \int_{T_i(\Delta)}^{T_i(\Delta) + t}\sigma_s(\rho_0(s)ds +\upsilon dB_s)\\
 & = &
 \sigma_{0}^{i} +\int_{0}^{t}\sigma_{s}^{i}
 (\rho_0(T_i(\Delta) + s)ds +\upsilon dB_{s}^{i})
 \quad {\rm with}\quad
 B^i := B_{T_i(\Delta) +\cdot} - B_{T_i(\Delta)}.
\end{eqnarray*}
Moreover, let $Z^i = (Z_{t}^{i})_{t\in [0,T]}$ be the process such that, for every $t\in [0,T]$,
\begin{eqnarray*}
 Z_{t}^{i} & := &
 \frac{\mathfrak c_H}{\upsilon}\int_{0}^{t}
 \frac{s^{1/2 - H}(t - s)^{1/2 - H}}{\sigma_{s}^{i}}d\sigma_{s}^{i}\\
 & = &
 \frac{1}{\upsilon}\int_{0}^{t}
 \ell(t,s)\rho_0(T_i(\Delta) + s)ds + M_{t}^{i}
 \quad {\rm with}\quad
 M^i :=
 \int_{0}^{.}\ell(.,s)dB_{s}^{i}.
\end{eqnarray*}
In the sequel, $\rho_0$ is $(T +\Delta)$-periodic, and then
\begin{displaymath}
Z_{t}^{i} =\frac{1}{\upsilon}\int_{0}^{t}J(\rho_0)(s)d\langle M^i\rangle_s + M_{t}^{i}
\textrm{$;$ }\forall t\in [0,T].
\end{displaymath}
Since $B$ has stationary increments, $M^1,\dots,M^N$ have the same distribution, but these Molchan martingales are not independent when $H > 1/2$. However, for any $i,k\in\{1,\dots,N\}$ such that $i < k$, and any $s,t\in [0,T]$ such that $s < t$,
\begin{eqnarray*}
 \mathbb E[B_{s}^{i}B_{t}^{k}] & = &
 \mathbb E[B_s(B_{t + T_{i,k}(\Delta)} - B_{T_{i,k}(\Delta)})]
 \quad {\rm with}\quad T_{i,k}(\Delta) = T_k(\Delta) - T_i(\Delta)\\
 & = &
 1/2[s^{2H} + (t + T_{i,k}(\Delta))^{2H} - (t + T_{i,k}(\Delta) - s)^{2H}
 - (s^{2H} + T_{i,k}(\Delta)^{2H} - (T_{i,k}(\Delta) - s)^{2H})]\\
 & = &
 1/2[(T_{i,k}(\Delta) + t)^{2H} + (T_{i,k}(\Delta) - s)^{2H}
 - (T_{i,k}(\Delta) + t - s)^{2H} - T_{i,k}(\Delta)^{2H}]\\
 & = &
 1/2T_{i,k}(\Delta)^{2H}
 [(1 + t/T_{i,k}(\Delta))^{2H} + (1 - s/T_{i,k}(\Delta))^{2H} -
 (1 + (t - s)/T_{i,k}(\Delta))^{2H} - 1]\\
 & = &
 2^{-1}H(2H - 1)T_{i,k}(\Delta)^{2H}[
 (t/T_{i,k}(\Delta))^2 + (s/T_{i,k}(\Delta))^2\\
 & &
 \hspace{4.5cm} - ((t - s)/T_{i,k}(\Delta))^2
 +\textrm o((1/T_{i,k}(\Delta))^2)]
 \quad {\rm when}\quad\Delta\rightarrow\infty\\
 & &
 \hspace{1cm}\sim_{\Delta\rightarrow\infty}
 H(2H - 1)\cdot st\cdot(k - i)^{2H - 2}(T +\Delta)^{2H - 2}.
\end{eqnarray*}
Since $(T +\Delta)^{2H - 2}\rightarrow 0$ when $\Delta\rightarrow\infty$, the more $\Delta$ is large, the more $B^i$ and $B^k$ (and then $M^i$ and $M^k$) {\it become} independent. So, for $\Delta$ large enough, if the constant $\upsilon$ is known, thanks to Subsection \ref{subsection_estimator_Q_0}, a {\it satisfactory} nonparametric estimator of $\rho_0$ is given by
\begin{displaymath}
\widehat\rho_{m,N}(t) :=
\upsilon\widehat Q_{m,N}(t)
\textrm{$;$ }
t\in [0,T],
\end{displaymath}
where $m\in\{1,\dots,N\}$,
\begin{displaymath}
\widehat Q_{m,N}(t) :=
\left\{
\begin{array}{lcl}
 \widehat J_{m,N}(t) & {\rm if} & H = 1/2\\
 \displaystyle{\overline{\mathfrak c}_Ht^{H - 1/2}
 \int_{0}^{t}(t - s)^{H - 3/2}s^{1 - 2H}\widehat J_{m,N}(s)ds}
 & {\rm if} & H > 1/2
\end{array}\right.
\textrm{$;$ }
t\in [0,T]
\end{displaymath}
and
\begin{displaymath}
\widehat J_{m,N} =\arg\min_{J\in\mathcal S_m}\left\{
\frac{1}{N}\sum_{i = 1}^{N}\left(
(2 - 2H)\int_{0}^{T}J(s)^2s^{1 - 2H}ds -
2\int_{0}^{T}J(s)dZ_{s}^{i}\right)\right\}.
\end{displaymath}
Assume that the $\varphi_j$'s belong to $\mathcal J$. Since
\begin{displaymath}
\|\widehat\rho_{m,N} -\rho_0\|_{2}^{2} =
\upsilon^2\|\widehat Q_{m,N} - Q_0\|_{2}^{2}
\quad {\rm with}\quad
Q_0 =\frac{\rho_0}{\upsilon},
\end{displaymath}
if $M^1,\dots,M^N$ were independent, then Proposition \ref{risk_bound_estimator_Q_0} would provide risk bounds on $\widehat\rho_{m,N}$ and on the adaptive estimator $\widehat\rho_{\widehat m,N}$ with $\widehat m$ selected in $\mathcal M_N$ via \eqref{model_selection_method}. Of course $M^1,\dots,M^N$ are not independent when $H > 1/2$, but the risk bounds of Proposition \ref{risk_bound_estimator_Q_0} remain relevant for $\Delta$ large enough as explained above.
\\
\\
Finally, the $(T +\Delta)$-periodicity condition on $\rho_0$ makes sense in the following special case; when $\rho_0$ is $T$-periodic and $\Delta =\delta T$ with $\delta\in\mathbb N^*$ large enough. As at Subsection \ref{subsection_Black_Scholes}, to assume that $\rho_0$ is $T$-periodic means that Model \eqref{fractional_stochastic_volatility_model} is appropriate for assets with a volatility process having {\it similar} trends on each interval $[(i - 1)T,iT]$, typically each day ($T = 24$h). When $T = 24$h, to assume $\delta$ large enough means to avoid enough days between two days during which the volatility process is observed in order to estimate $\rho_0$ with our method.
%

% Section : Numerical experiments.

%
\section{Numerical experiments}\label{section_numerical_experiments}
This section deals with numerical experiments in Model \eqref{main_equation} when $M$ is the Molchan martingale, and in the non-autonomous Black-Scholes model.
%

% Subsection : Experiments in Model (main_equation) driven by the Molchan martingale.

%
\subsection{Experiments in Model \eqref{main_equation} driven by the Molchan martingale}
Some numerical experiments on our estimation method of $J_0$ in Equation \eqref{main_equation} are presented in this subsection when $M$ is the Molchan martingale:
\begin{displaymath}
M_t =
\int_{0}^{t}\ell(t,s)dB_s =
(2 - 2H)^{1/2}
\int_{0}^{t}s^{1/2 - H}dW_s
\textrm{$;$ }
t\in [0,1]
\end{displaymath}
with $H\in\{0.6,0.9\}$ and $W$ the Brownian motion driving the Mandelbrot-Van Ness representation of the fractional Brownian motion $B$. The estimation method investigated on the theoretical side at Section \ref{risk_bounds_model_selection_section} is implemented here for the three following examples of functions $J_0$:
\begin{displaymath}
J_{0,1} : t\in [0,1]\mapsto 10t^2,\quad
J_{0,2} : t\in (0,1]\mapsto 10(-\log(t))^{1/2}\quad {\rm and}\quad
J_{0,3} : t\in (0,1]\mapsto 20t^{-0.05}.
\end{displaymath}
These functions belong to $\mathbb L^2([0,1],d\langle M\rangle_t)$ as required. Indeed, on the one hand, $J_{0,1}$ is continuous on $[0,1]$ and
\begin{eqnarray*}
 -\int_{0}^{1}\log(t)d\langle M\rangle_t & = &
 -(2 - 2H)\int_{0}^{1}\log(t)t^{1 - 2H}dt\\
 & = &
 \lim_{\varepsilon\rightarrow 0^+}
 \log(\varepsilon)\varepsilon^{2 - 2H} +
 \int_{0}^{1}t^{1 - 2H}dt =
 \frac{1}{2 - 2H} <\infty.
\end{eqnarray*}
On the other hand, for every $\alpha\in (0,1/2)$ such that $H\in (1/2,1 -\alpha)$,
\begin{eqnarray*}
 \int_{0}^{1}t^{-2\alpha}d\langle M\rangle_t & = &
 (2 - 2H)\int_{0}^{1}t^{1 - 2\alpha - 2H}dt\\
 & = &
 \frac{2 - 2H}{2(1 -\alpha - H)}\left(1 -
 \lim_{\varepsilon\rightarrow 0^+}\varepsilon^{2(1 -\alpha - H)}\right) =
 \frac{1 - H}{1 -\alpha - H} <\infty.
\end{eqnarray*}
Since for every $t\in (0,1]$, $J_{0,3}(t) = 20t^{-\alpha}$ with $\alpha = 0.05$, and since $H\in\{0.6,0.9\}\subset (0.5,0.95)$ in our numerical experiments, $J_{0,3}$ is square-integrable with respect to $d\langle M\rangle_t$.\\
Our adaptive estimator is computed for $J_0 = J_{0,1}$, $J_{0,2}$ and $J_{0,3}$ on $N = 100$ paths of the process $Z$ observed along the dissection $\{l/n\textrm{$;$ }l = 1,\dots,n\}$ of $[1/n,1]$ with $n = 5000$, when $(\varphi_1,\dots,\varphi_m)$ is the $m$-dimensional trigonometric basis for every $m\in\{2,\dots,12\}$. Note that $n$ is of order $N^2$ as suggested in the remark following Proposition \ref{risk_bound_discrete}. This experiment is repeated $100$ times, and the means and the standard deviations of the MISE of $\widehat J_{\widehat m,N,n}$ (see Subsection \ref{projection_LS_estimator_definition_subsection}) are stored in Table \ref{table_MISE}.
\begin{table}[!h]
\begin{center}
\begin{tabular}{|l||c|c|c||c|c|c|}
 \hline
 $J_0; H$ & $J_{0,1}; 0.6$ & $J_{0,2}; 0.6$ & $J_{0,3}; 0.6$ &
 $J_{0,1}; 0.9$ & $J_{0,2}; 0.9$ & $J_{0,3}; 0.9$\\
 \hline
 \hline
 Mean MISE & 0.047 & 0.103 & 0.076 & 0.135 & 0.300 & 0.287\\
 \hline
 StD MISE & 0.031 & 0.029 & 0.033 & 0.076 & 0.118 & 0.084\\
 \hline
\end{tabular}
\medskip
\caption{Means and StD of the MISE of $\widehat J_{\widehat m,N,n}$ (100 repetitions).}\label{table_MISE}
\end{center}
\end{table}
Moreover, for $H = 0.6$, $10$ estimations (dashed black curves) of $J_{0,1}$, $J_{0,2}$ and $J_{0,3}$ (red curves) are respectively plotted on Figures \ref{J_1_plot}, \ref{J_2_plot} and \ref{J_3_plot}.
\begin{figure}[!h]
\centering
\includegraphics[scale=0.4]{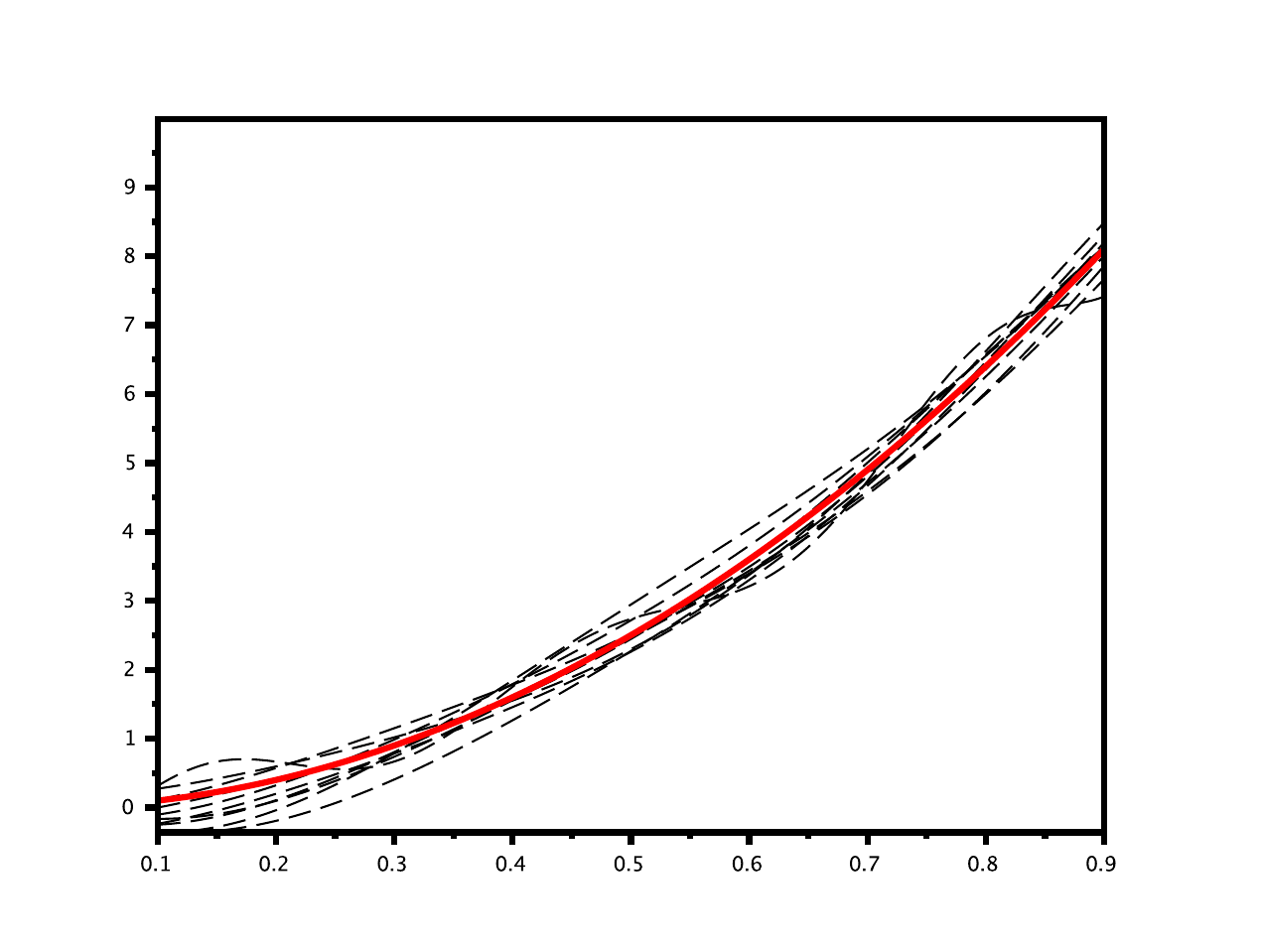} 
\caption{Plots of $J_{0,1}$ and of $10$ adaptive estimations when $H = 0.6$ ($\overline{\widehat m} = 5.4$).}
\label{J_1_plot}
\end{figure}
\begin{figure}[!h]
\centering
\includegraphics[scale=0.4]{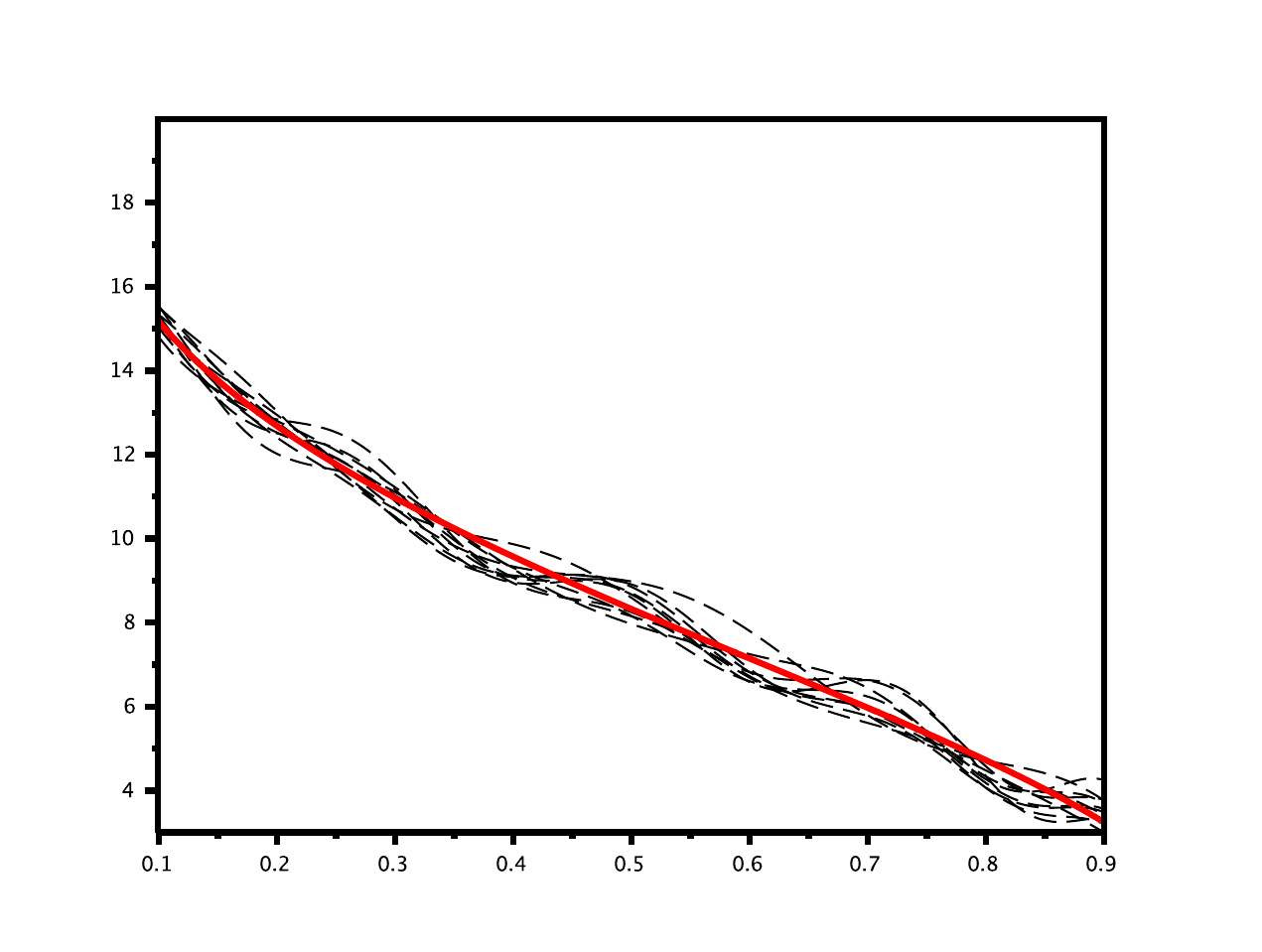} 
\caption{Plots of $J_{0,2}$ and of $10$ adaptive estimations when $H = 0.6$ ($\overline{\widehat m} = 11.2$).}
\label{J_2_plot}
\end{figure}
\begin{figure}[!h]
\centering
\includegraphics[scale=0.4]{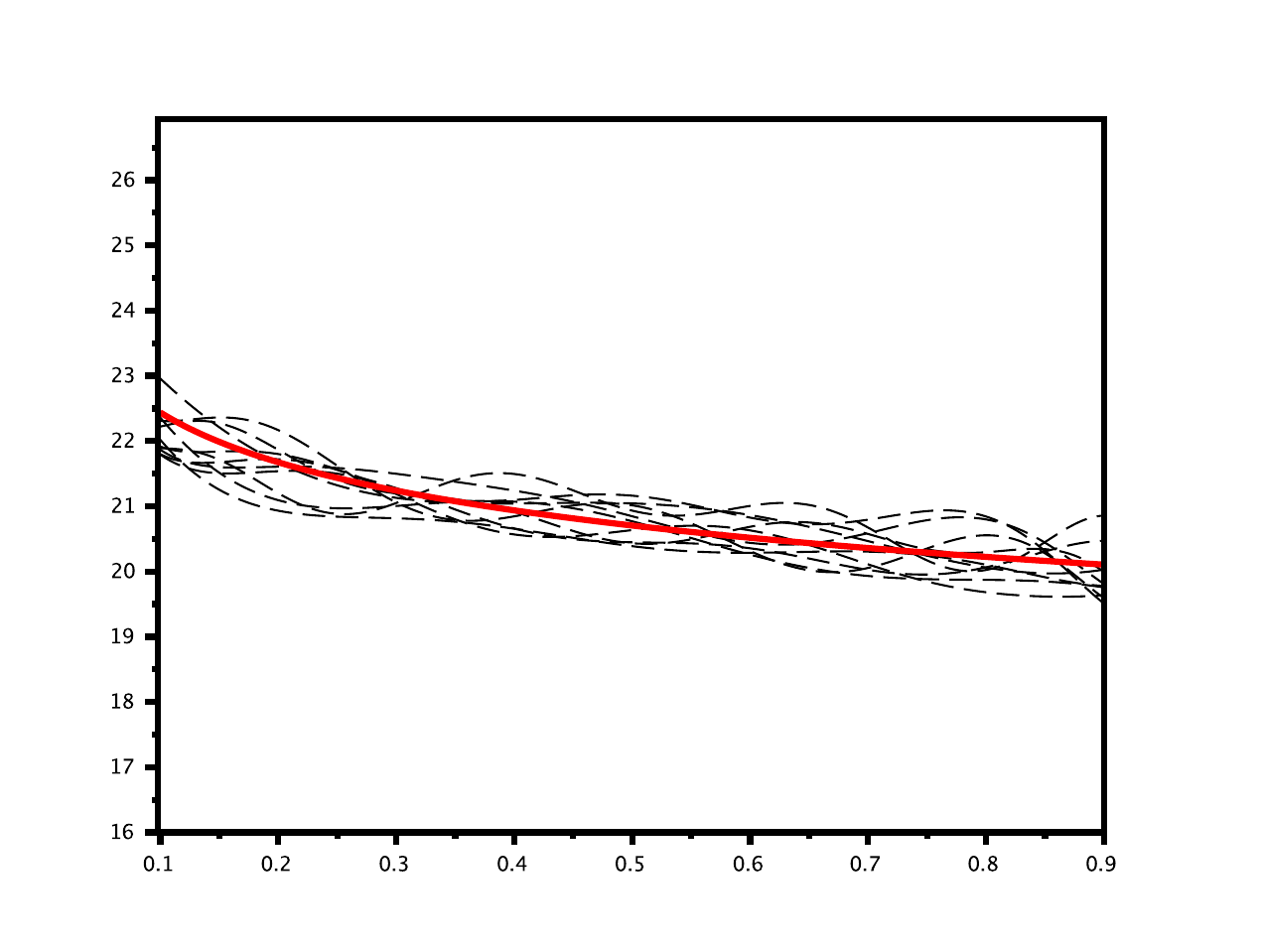} 
\caption{Plots of $J_{0,3}$ and of $10$ adaptive estimations when $H = 0.6$ ($\overline{\widehat m} = 8.2$).}
\label{J_3_plot}
\end{figure}
On average, the MISE of our adaptive estimator is lower for the three examples of functions $J_0$ when $H = 0.6$ than when $H = 0.9$. The standard deviation of the MISE of $\widehat J_{\widehat m,N,n}$ is also higher when $H = 0.9$. For $H = 0.6$, our estimation method seems stable in the sense that the standard deviation of the MISE of our adaptive estimator is almost the same for the three examples of functions $J_0$. This can be observed on Figures \ref{J_1_plot}, \ref{J_2_plot} and \ref{J_3_plot}. For $H = 0.9$, our estimation method seems less stable. Finally, for both $H = 0.6$ and $H = 0.9$, on average, the MISE of $\widehat J_{\widehat m,N,n}$ is higher for $J_{0,2}$ than for $J_{0,1}$ and $J_{0,3}$. This is probably related to the fact that $J_{0,2}(t)$ goes faster to infinity when $t\rightarrow 0^+$ than $J_{0,3}(t)$, and of course than $J_{0,1}(t)$ which doesn't go. In conclusion, the numerical experiments show that when $Z$ is driven by the Molchan martingale, our estimation method of $J_0$ is satisfactory on several types of functions, but the MISE of $\widehat J_{\widehat m,N,n}$ seems to increase when $H$ is near to $1$.
%

% Subsection : Experiments in the non-autonomous Black-Scholes model.

%
\subsection{Experiments in the non-autonomous Black-Scholes model}
Some numerical experiments on our estimation method of $b_0$ in the non-autonomous Black-Scholes model \eqref{Black_Scholes_model} are presented in this subsection on simulated prices datasets with $S_0 = 10$, $\sigma\in\{0.2,1\}$ and $b_0 =\mathfrak b$, where
\begin{displaymath}
\mathfrak b(t) :=
\sin(2\pi t) +\cos(2\pi t)
\textrm{$;$ }
\forall t\in\mathbb R_+.
\end{displaymath}
The function $\mathfrak b$ is estimated on $[0,1]$ (1 day), but the prices process $S$ of the asset is simulated on $N = 100$ days via the non-autonomous Black-Scholes model \eqref{Black_Scholes_model}. As explained at Subsection \ref{subsection_Black_Scholes}, here, $Z^1,\dots,Z^N$ are obtained via \eqref{process_Z_Black_Scholes} from the i.i.d. processes $S^1,\dots,S^N$ defined by
\begin{displaymath}
S_{t}^{i} := S_{i - 1 + t}
\textrm{$;$ }\forall t\in [0,1].
\end{displaymath}
Our adaptive estimator is computed for $S_0 = 10$ and $\sigma\in\{0.2,1\}$ on the paths of $Z^1,\dots,Z^N$ obtained from one path of the prices process observed along the dissection $\{100l/n\textrm{$;$ }l = 0,\dots,n\}$ of $[0,100]$ with $n = N^2 = 10000$, when $(\varphi_1,\dots,\varphi_m)$ is the $m$-dimensional trigonometric basis for every $m\in\{2,\dots,12\}$. Note that for $\sigma = 0.2$, the path of $S$ is plotted on Figure \ref{BS_02}, and the associated paths of $Z^1,\dots,Z^N$ are plotted on Figure \ref{Z_02}.
\begin{figure}[!h]
\centering
\includegraphics[scale=0.4]{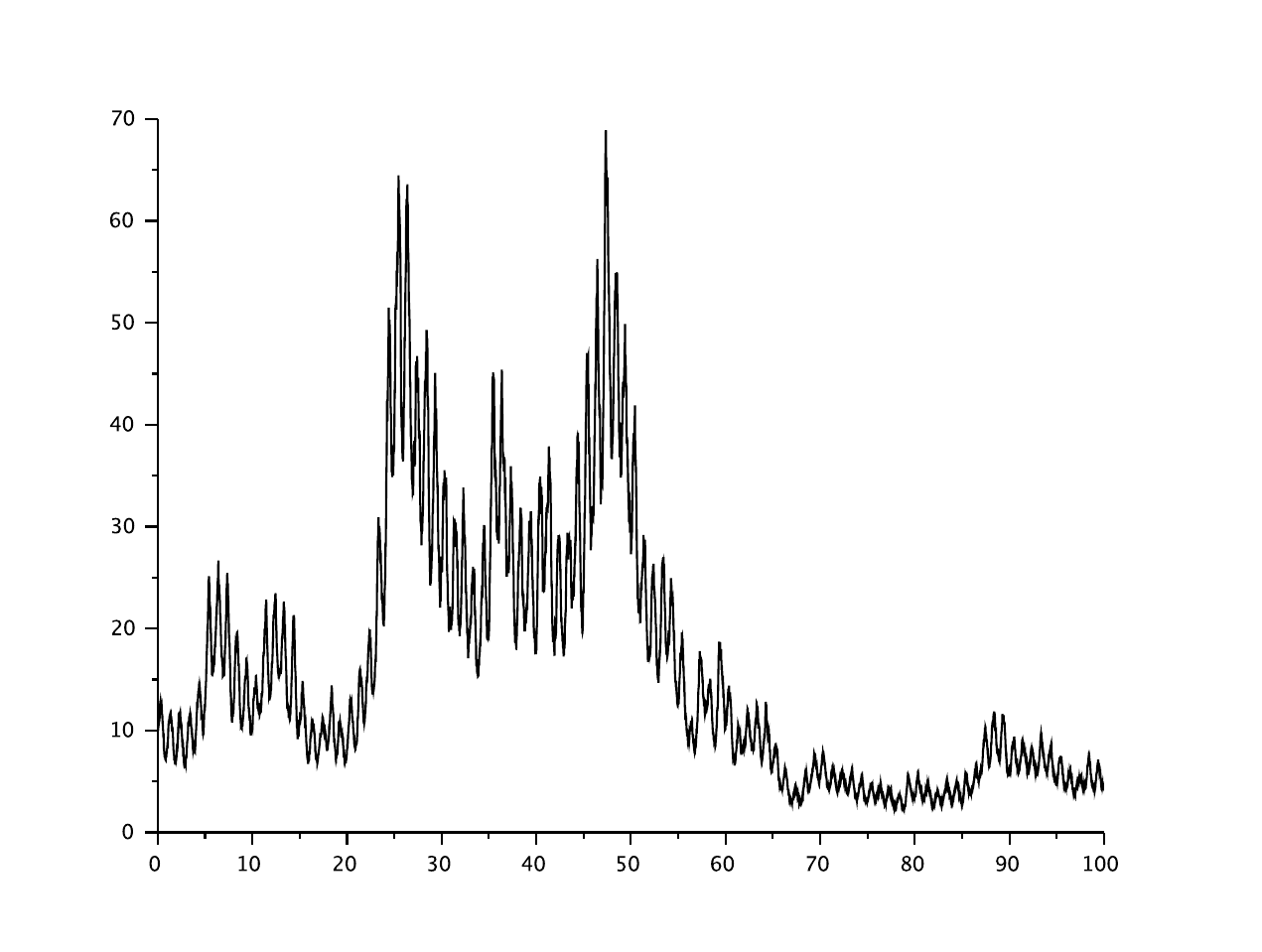}
\caption{Plot of one path of the non-autonomous Black-Scholes model with $\sigma = 0.2$.}
\label{BS_02}
\end{figure}
\begin{figure}[!h]
\centering
\includegraphics[scale=0.4]{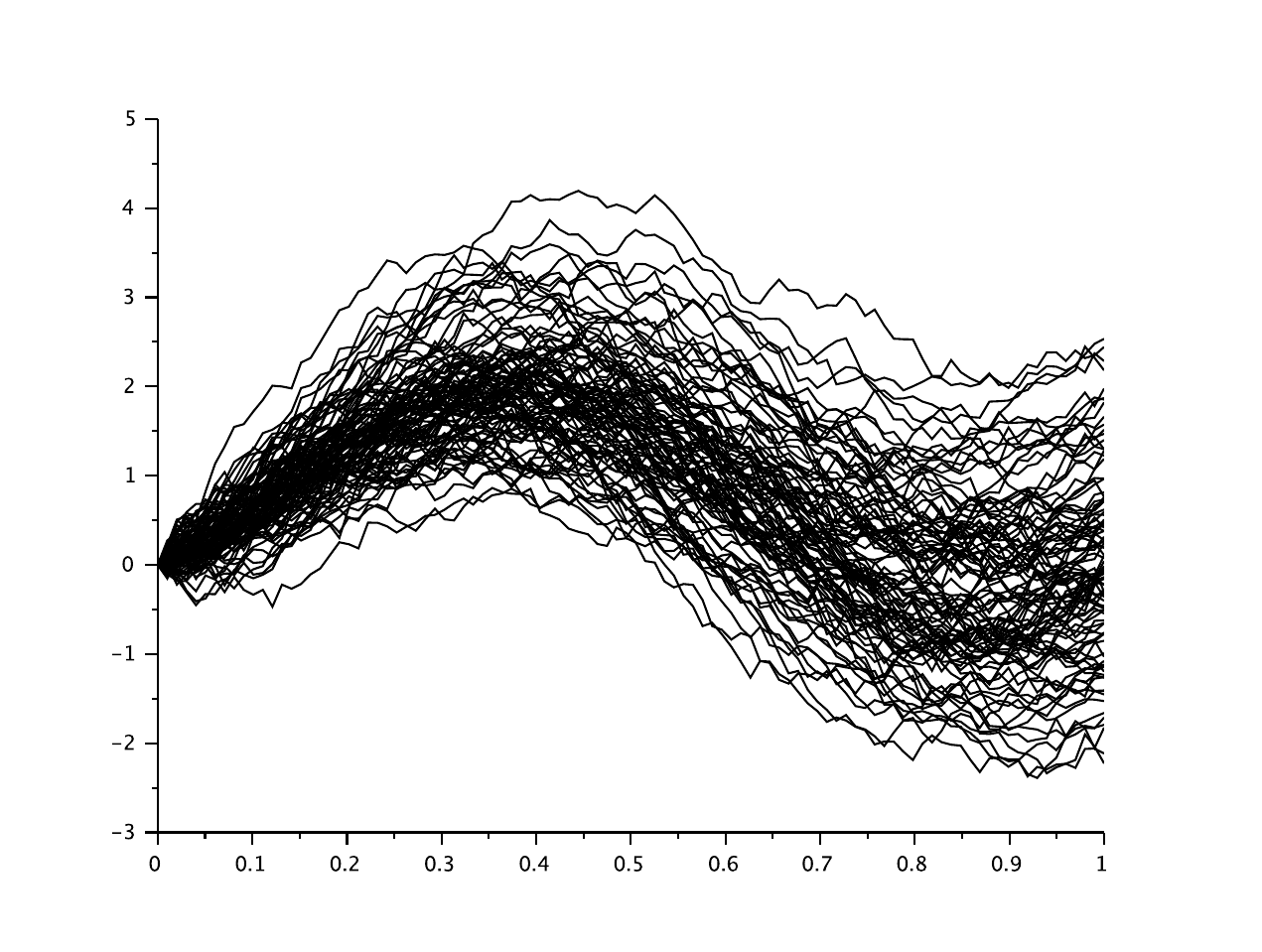}
\caption{Plots of the paths of $Z^1,\dots,Z^N$ associated to the path of $S$ at Figure \ref{BS_02}.}
\label{Z_02}
\end{figure}
This experiment is repeated 100 times, and the means and the standard deviations of the MISE of $\widehat b_{\widehat m,N,n} :=\sigma\widehat J_{\widehat m,N,n}$ are stored in Table \ref{table_MISE_BS}.
\begin{table}[!h]
\begin{center}
\begin{tabular}{|l||c|c|c||c|c|c|}
 \hline
 $\sigma$ & $0.2$ & $1$\\
 \hline
 \hline
 Mean MISE & 0.002 & 0.042\\
 \hline
 StD MISE & 0.002 & 0.046\\
 \hline
\end{tabular}
\medskip
\caption{Means and StD of the MISE of $\widehat b_{\widehat m,N,n}$ (100 repetitions).}\label{table_MISE_BS}
\end{center}
\end{table}
Moreover, for $\sigma = 0.2$ and $\sigma = 1$, 10 estimations (dashed black curves) of $b_0 =\mathfrak b$ (red curve) are respectively plotted on Figures \ref{BS_function_02} and \ref{BS_function_1}.
\begin{figure}[!h]
\centering
\includegraphics[scale=0.4]{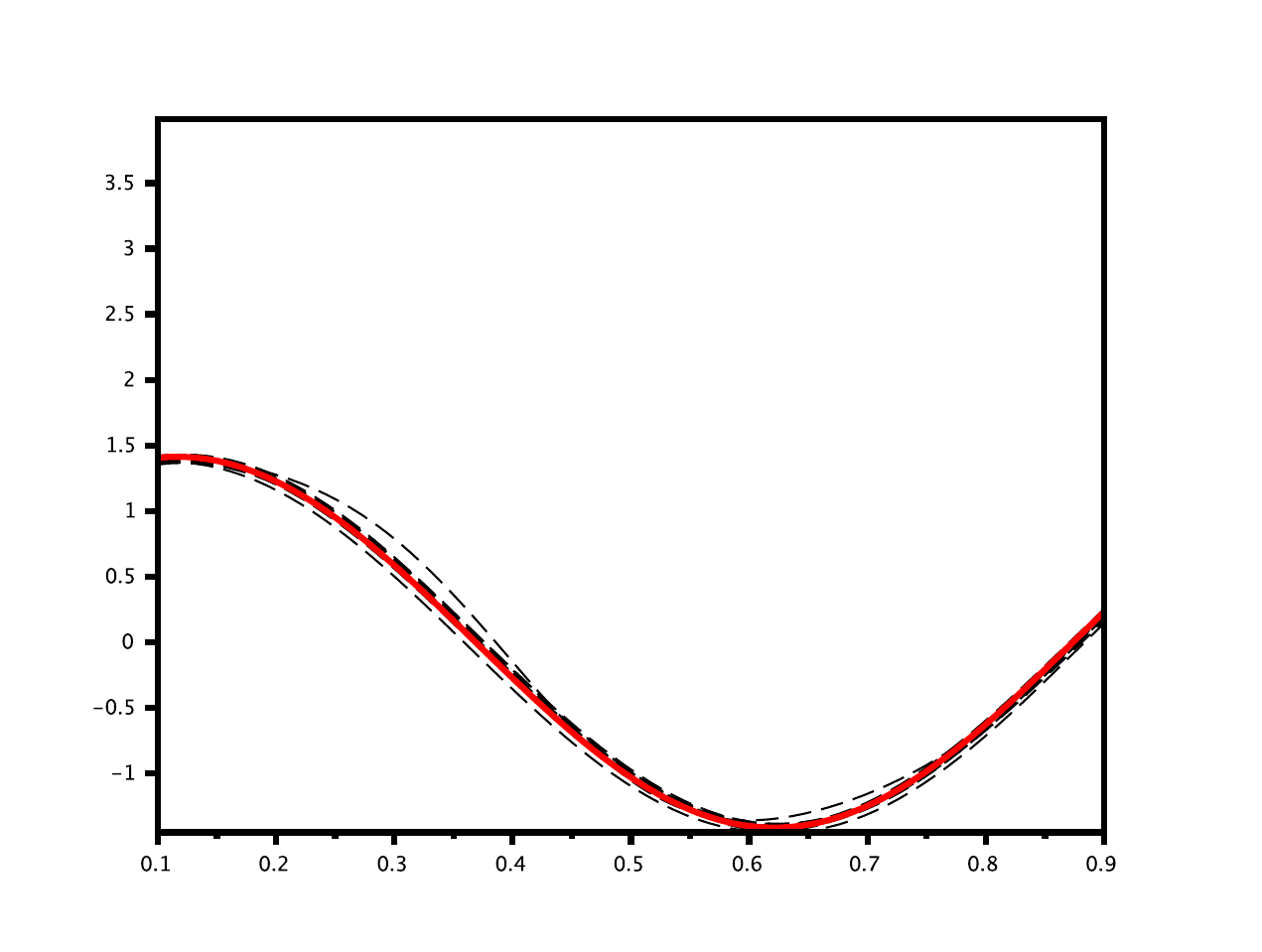}
\caption{Plots of $\mathfrak b$ and of $10$ adaptive estimations when $\sigma = 0.2$ ($\overline{\widehat m} = 3.3$).}
\label{BS_function_02}
\end{figure}
\begin{figure}[!h]
\centering
\includegraphics[scale=0.4]{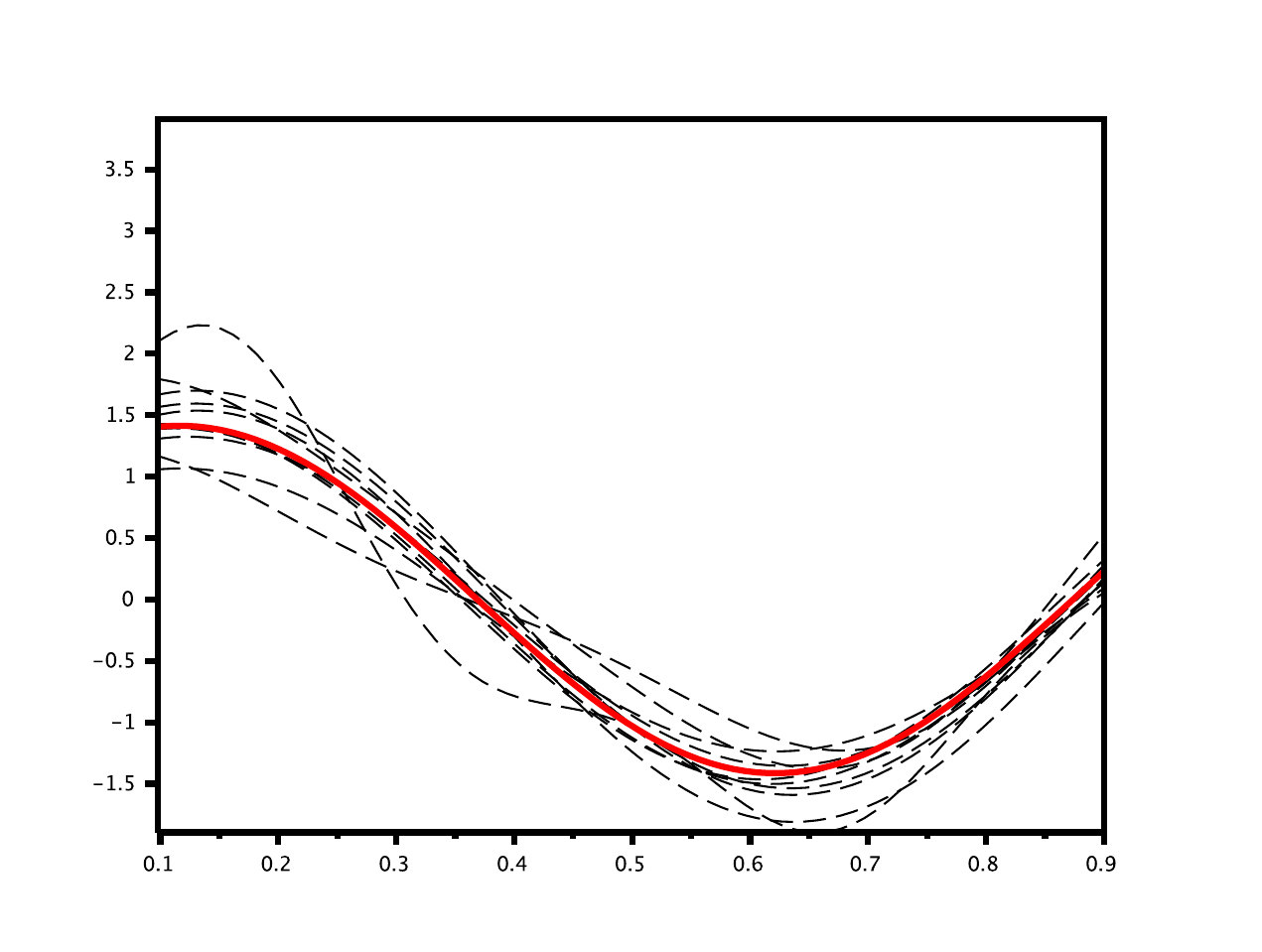}
\caption{Plots of $\mathfrak b$ and of $10$ adaptive estimations when $\sigma = 1$ ($\overline{\widehat m} = 3.5$).}
\label{BS_function_1}
\end{figure}
On average, the MISE of our adaptive estimator is obviously lower when $\sigma = 0.2$ than when $\sigma = 1$, but it remains globally small. Assume now that $\sigma$ is unknown as in practice. Then, it is estimated directly on the observed path of $S$ on $[0,100]$ by
\begin{displaymath}
\widehat\sigma_{N,n} :=
\sqrt{\frac{1}{T_{N,n}}\sum_{l = 0}^{n - 1}
[\log(S_{(l + 1)N/n}) -\log(S_{lN/n})]^2}
\quad{\rm with}\quad T_{N,n} = nN
\end{displaymath}
as usual (see Genon-Catalot \cite{GC12}, Subsubsection 3.2.2). The same experiment is then repeated on
\begin{displaymath}
\widetilde b_{\widehat m,N,n} :=\widehat\sigma_{N,n}\widehat J_{\widehat m,N,n}
\end{displaymath}
instead of $\widehat b_{\widehat m,N,n}$. The means and the standard deviations of the MISE of $\widetilde b_{\widehat m,N,n}$ remain of same order (see Table \ref{table_MISE_BS_sigma_unknown}).
\begin{table}[!h]
\begin{center}
\begin{tabular}{|l||c|c|c||c|c|c|}
 \hline
 $\sigma$ (mean $\widehat\sigma_{N,n}$) & $0.2$ (0.223) & $1$ (1.005)\\
 \hline
 \hline
 Mean MISE & 0.001 & 0.042\\
 \hline
 StD MISE & 0.001 & 0.038\\
 \hline
\end{tabular}
\medskip
\caption{Means and StD of the MISE of $\widetilde b_{\widehat m,N,n}$ (100 repetitions).}\label{table_MISE_BS_sigma_unknown}
\end{center}
\end{table}
\newline
\newline
{\bf Acknowledgments.} Thank you to Fabienne Comte for her valuable comments on this paper.
%

% References.

%

%
\end{document}